\documentclass[10pt]{article}
\usepackage{amscd}
\usepackage{amsfonts}
\usepackage{amsmath}
\usepackage{amssymb}
\usepackage{amsthm}
\usepackage{bbm}
\usepackage{CJK}
\usepackage{fancyhdr}
\usepackage{graphicx}
\usepackage{hyperref}
\usepackage{indentfirst}
\usepackage{latexsym}
\usepackage{mathrsfs}
\usepackage{xypic}
\usepackage[all]{xy}           
\usepackage{amsfonts,latexsym}
\usepackage{xspace}
\usepackage{epsfig}
\usepackage{float}
\usepackage{txfonts}
\usepackage{amssymb}
\usepackage{graphicx}
\usepackage{amsmath}
\usepackage{xcolor}
\usepackage[top=1in,bottom=1in,left=1.25in,right=1.25in]{geometry}
\usepackage{epsf,graphicx,epsfig}
\usepackage{amsmath,latexsym,amssymb,amsthm}
\usepackage[nospace,noadjust]{cite}
\usepackage{textcomp}
\usepackage{setspace,cite}
\usepackage{stmaryrd}
\usepackage{tikz-cd} 
\usepackage{tikz}
\usepackage{cancel}
\usetikzlibrary{calc,shapes,arrows,decorations.markings,decorations.pathmorphing}

\usepackage{color}
\usepackage{url}
\usepackage{enumerate}
\usepackage[mathscr]{euscript}

\swapnumbers
    \newtheorem{thm}{Theorem}[section]
    \newtheorem{prop}[thm]{Proposition}
    
    \newtheorem*{Proof*}{Proof}

    \newtheorem{subsec}[thm]{}
\theoremstyle{Definition}
    \newtheorem{Def}[thm]{Definition}
        \newtheorem{Rem}[thm]{Remark}
    \newtheorem{Exam}[thm]{Example}

    \newtheorem{notation}[thm]{Notation}

\theoremstyle{remark}

\renewcommand{\ker}{\operatorname{Ker}}

\usepackage{amssymb}

\usepackage{hyperref}
\hypersetup{
	colorlinks,
	citecolor=blue,
	filecolor=black,
	linkcolor=blue,
	urlcolor=black
}
\tikzset{
  curve/.style={
    settings={#1},
    to path={
      (\tikztostart)
      .. controls ($(\tikztostart)!\pv{pos}!(\tikztotarget)!\pv{height}!270:(\tikztotarget)$)
      and ($(\tikztostart)!1-\pv{pos}!(\tikztotarget)!\pv{height}!270:(\tikztotarget)$)
      .. (\tikztotarget)\tikztonodes
    },
  },
  settings/.code={%
    \tikzset{quiver/.cd,#1}%
    \def\pv##1{\pgfkeysvalueof{/tikz/quiver/##1}}%
  },
  quiver/.cd,
  pos/.initial=0.35,
  height/.initial=0,
}

\date{}
\begin{document}
\renewcommand{\baselinestretch}{1.2}
\renewcommand{\arraystretch}{1.0}
\title{\bf Cohomology and Extensions of $C_p$-Green Functors of Lie Type}
\date{}
\author{{\bf Tarik Anowar$^{1}$, Satyendra Kumar Mishra$^{2}$, Ripan Saha$^{3}$\footnote
        { Corresponding author:~~Email: ripanjumaths@gmail.com}}\\
{\small 1.  Department of Mathematics, Raiganj University, Raiganj 733134, West Bengal, India}\\
{\small 2. Department of Mathematical Sciences, Indian Institute of Technology (BHU), Varanasi-221005, India}\\
{\small 3. Department of Mathematics, Raiganj University, Raiganj 733134, West Bengal, India}
}
 \maketitle
\begin{center}
\begin{minipage}{12.cm}
\begin{center}{\bf ABSTRACT}\end{center}
We develop a theory of \(C_p\)-Green functors of Lie type, unifying the axiomatic framework of Green functors with the structure of Lie algebras under the action of a cyclic group \(C_p\) of prime order. Extending classical notions from representation theory and topology, we define tensor and exterior products, introduce an equivariant Chevalley–Eilenberg cohomology, and construct cup products that endow the cohomology with a graded Green functor of Lie type structure. A key result establishes a correspondence between equivalence classes of singular extensions and second cohomology groups, generalizing classical Lie algebra extension theory to the equivariant setting. This framework enriches the toolkit for studying equivariant algebraic structures and paves the way for further applications in deformation theory, homotopical algebra, and representation theory.

\medskip

\medskip

\end{minipage}
\end{center}
\normalsize\vskip0.5cm
 \section{Introduction}
Mackey functors originated with George Mackey’s work on group representations, formalizing systems of modules over subgroups connected by restriction and induction maps. James A. Green \cite{Gr} expanded this framework by introducing multiplicative structures (Green functors) to study modular representations, particularly vertices and sources of indecomposable modules. Andreas Dress \cite{Dr} axiomatized Mackey functors in the 1970s, linking them to Burnside rings and algebraic K-theory. In the 1980s, Lewis, May, and McClure \cite{LMM} embedded Green functors into equivariant topology, where they became central to understanding equivariant cohomology theories like Bredon cohomology \cite{Br}. In 1990, Thévenaz and Webb \cite{Th} develop cohomology theories for Mackey functors, connecting them to derived functors and subgroup lattices. Modern advancements extend these ideas to derived categories, tensor-triangular geometry, and motivic homotopy theory, cementing Mackey and Green functors as foundational tools in representation theory, topology, and equivariant mathematics.

The study of equivariant cohomology provides a powerful framework for understanding algebraic structures endowed with group actions, bridging representation theory, homological algebra, and topology. Central to this framework are Mackey and Green functors, which systematically encode algebraic data across subgroups of a finite group while respecting restriction, transfer, and conjugation maps. These functors have found profound applications in equivariant topology\cite{LMM} and modular representation theory. In \cite{Ko}, the author studied an aspect of equivariant associative algebras using Mackey functor and the notion of oriented group action. In this paper, we extend these ideas to the setting of Lie algebras, introducing and analyzing \emph{\(C_p\)-Green functors of Lie type}---a novel structure that combines the axiomatic framework of Green functors with the algebraic richness of Lie theory, particularly under the action of a cyclic group \(C_p\) of prime order. In this paper, we only deal with the cyclic group \( C_p \) of prime order \( p \), as such groups have only two subgroups, making the resulting structure simpler and more tractable. This serves as a natural first step toward the study of general \( G \)-Green functors of Lie type and the development of a full-fledged cohomology theory.

Lie algebras equipped with group actions arise naturally in deformation theory, physics, and geometry. However, their equivariant cohomology, especially in the context of Mackey functors, remains underexplored. Our work aims to fill this gap by developing a categorical and cohomological theory for Lie-type Green functors equivariant under \(C_p\). Specifically, we establish how the interplay between Lie brackets and the functorial operations (restriction, transfer, conjugation) enriches both the structure theory and cohomological invariants of these objects.

The paper begins by recalling foundational concepts, including Mackey functors, Green functors, and their module categories. We then define \(C_p\)-Green functors of Lie type, requiring that the evaluations at the trivial subgroup \(e\) and the group \(C_p\) themselves carry Lie algebra structures, compatible with the functorial maps. Morphisms in this category are pairs of Lie algebra homomorphisms that commute with the restriction, transfer, and conjugation maps, ensuring coherence across subgroups.

Key constructions include the tensor and exterior products of \(C_p\)-Green functors of Lie type. These operations are carefully designed to respect the Mackey functor axioms while preserving antisymmetry and Jacobi identities. We further introduce the equivariant Hom complex, \(C_p\text{-}\mathrm{HOM}(\wedge^n L_1, L_2)\), which generalizes alternating multilinear maps to the Mackey functor setting.

Building on this, we define the \emph{Chevalley-Eilenberg cohomology} for \(C_p\)-Green functors of Lie type, demonstrating that it naturally inherits the structure of a Mackey functor. A central result is the classification of singular extensions of \(C_p\)-Green functors of Lie type by bimodules, establishing a one-to-one correspondence between equivalence classes of such extensions and the second Chevalley-Eilenberg cohomology group \(H^2_{\mathrm{Lie}_{C_p}}(L, M)\).

The paper is structured as follows: In Section \ref{sec2}, we recall Mackey and Green functors, the notion of pairing. In Section \ref{sec3}, we introduce \(C_p\)-Green functors of Lie type and their morphisms. In Section \ref{sec4}, we develop tensor and exterior products. Section \ref{sec5} constructs the Chevalley-Eilenberg cohomology. In Section \ref{sec6}, we define cup-product on the cochain complex Mackey functors and show that graded cohomology \( H^{*}(L,M) \) becomes a \( C_{P} \)-Green functor of graded Lie type under this cup-product. Finally, in Section \ref{sec7}, we classify singular extensions, generalizing the classical Lie algebra extension theory to the equivariant setting.

Our results underscore the versatility of Green functors in equivariant algebra and provide new tools for studying Lie structures in categories enriched with group actions. This work lays the groundwork for future explorations in equivariant deformation theory based on the work of Gerstenhaber \cite{G63, G64}, homotopical algebra, and applications to geometric contexts where symmetries and noncommutative structures intertwine.
\section{Preliminary}\label{sec2}
In this section, we recall basic definitions and results following \cite{Be,Bl, Dr,Gr,E87, E91, We, Weibel, luca}, which will be used throughout the paper. 

Fix a commutative ring \( R \) and a finite group \( G \), and denote by \( S(G) \) the set of all subgroups of \( G \). For \( H \in S(G) \) and \( g \in G \), let \( ^{g}H = gHg^{-1} \). If \( H \in S(G) \) and \( L, K \in S(H) \), let \( [L \setminus H / K] \) be a set of representatives of double cosets \( LhK \) with \( h \in H \). Finally, if \( K \leq H \leq G \), let \( H/K \) be a set of representatives of cosets \( hK \) with \( h \in H \). In this paper, we always assume $R$ is a commutative ring with characteristics $0$.

\begin{Def}
A \( G \)-Mackey functor \( M \) (over \( R \)) is a functor from the category of subgroups \( S(G) \) of \( G \) to the category of \( R \)-modules, along with three families of \( R \)-linear morphisms. The families \( r_{K}^{H} : M(H) \rightarrow M(K) \) of restrictions and \( t_{K}^{H} : M(K) \rightarrow M(H) \) of transfers are indexed on the collection of \( H, K \in S(G) \) such that \( K \leq H \). The family \( c_{g} = c_{g,H} : M(H) \rightarrow M(^{g}H) \) of conjugation maps is indexed on \( H \in S(G) \) and \( g \in G \). For all \( g, h \in G \) and \( H, K, L \in S(G) \), these maps must satisfy the following conditions:

\begin{enumerate}
    \item[(i)] If \(L \leq K \leq H \leq G \), then 
    \[
    r_{L}^{K} \circ r_{K}^{H} = r_{L}^{H}, \quad t_{K}^{H} \circ t_{L}^{K} = t_{L}^{H}.
    \]
    
    \item[(ii)] Identity properties:
    \[
    r_{H}^{H} = \text{id}_{M(H)}, \quad t_{H}^{H} = \text{id}_{M(H)}.
    \]
    
    \item[(iii)] Conjugation composition:
    \[
    c_{gh} = c_{g} \circ c_{h}.
    \]
    
    \item[(iv)] If \( h \in H \), then \( c_{h} : M(H) \rightarrow M(H) \) is the identity map.
    
    \item[(v)] Compatibility of conjugation with restrictions and transfers:
    \[
    c_{g} \circ r_{K}^{H} = r_{^{g}K}^{^{g}H} \circ c_{g}, \quad c_{g} \circ t_{K}^{H} = t_{^{g}K}^{^{g}H} \circ c_{g}.
    \]
    
    \item[(vi)] Mackey Axiom: If \( L, K \leq H \), then
    \[
    r_{L}^{H} \circ t_{K}^{H} = \sum_{h \in [L \setminus H / K]} t_{L \cap ^{h}K}^{L} \circ r_{L \cap ^{h}K}^{^{h}K} \circ c_{h}.
    \]
\end{enumerate}

\end{Def}

\begin{Exam}
Let \( M \) be a left \( R[G] \)-module. The fixed point functor \( FP_{M} \) is defined by
\[
FP_{M}(H) = M^{H} = \{m \in M \mid h \cdot m = m, \text{ for all } h \in H \}.
\]
Let \(H\) and \(K\) be two subgroups of $G$ such that \(K\leq H\). Now, the restriction map \( r^{H}_{K} : M^{H} \rightarrow M^{K} \) is the inclusion of fixed points. The transfer map \( t^{H}_{K} : M^{K} \rightarrow M^{H} \) is the relative transfer map, defined by
\[
t^{H}_{K}(m) = \sum_{h \in H/K} h \cdot m,
\]
for \( m \in M^{K} \). The conjugation maps are defined by
\[
c_{g}(m) = g \cdot m,
\]
for all \( g \in G \) and \( m \in M \).

It can be easily verified that the fixed point functor \( FP_{M} \) is a Mackey functor.
\end{Exam}

Let \( R[W_G H] \) be the group ring of the Weyl group \( W_G H \), where \( W_G H = N_G(H)/H \) is the normalizer of \( H \) in \( G \) modulo \( H \). 

Let \( M \) and \( N \) be two Mackey functors. A morphism of two Mackey functors \( f : M \rightarrow N \) is a family of \( R \)-module homomorphisms \( f(H) : M(H) \rightarrow N(H) \), for \( H \in \mathcal{S}(G) \), which commute with the three maps: restriction, transfer, and conjugation. Since the morphism \( f \) commutes with conjugation, this implies that \( f(H) \) is an \( R[W_G H] \)-map. 

The category of Mackey functors for \( G \) over \( R \) is denoted by \(\text{Mack}_{R}(G)\). The set of morphisms between any two Mackey functors \( M \) and \( N \) is written as \(\text{Hom}_{\text{Mack}_{R}(G)}(M, N)\). 

It is easy to see that \(\text{Mack}_{R}(G)\) is an abelian category.

\subsection{Pairing Notation}
Here we recall the concept of the product in the context of a pairing. Let \( M \), \( N \), and \( L \) be \( G \)-Mackey functors over the ring \( R \). A pairing is a family of \( R \)-linear maps 
$$
\cdot : M(H) \times N(H) \rightarrow L(H) \quad \text{defined as} \quad (x, y) \mapsto x \cdot y
$$
such that for \( H, K \in \mathcal{S}(G) \) with \( K \leq H \), the following conditions hold:
\begin{enumerate}
    \item[(vii)] \( r^{H}_{K}(x \cdot y) = r^{H}_{K}(x) \cdot r^{H}_{K}(y) \), for \( x \in M(H) \) and \( y \in N(H) \),
    \item[(viii)] \( c_{g}(x \cdot y) = c_{g}(x) \cdot c_{g}(y) \), for \( x \in M(H) \), \( y \in N(H) \), and \( g \in G \),
    \item[(ix)] (Frobenius axiom)
    $$
    t^{H}_{K}(x \cdot r^{H}_{K}(y')) = t^{H}_{K}(x) \cdot y', \quad \text{for } x \in M(K), \, y' \in N(H),
    $$
    $$
    t^{H}_{K}(r^{H}_{K}(x') \cdot y) = x' \cdot t^{H}_{K}(y), \quad \text{for } x' \in M(H), \, y \in N(K).
    $$
\end{enumerate}

\begin{Def}
A Green functor \( A \) for the group \( G \) over a ring \( R \) is a Mackey functor \( A \) equipped with a pairing 
$$
A \times A \rightarrow A
$$
such that, for each \( H \in \mathcal{S}(G) \), the \( R \)-linear map 
$$
A(H) \times A(H) \rightarrow A(H)
$$
makes \( A(H) \) into an associative \( R \)-algebra with unity \( 1_{A(H)} \). Moreover, this structure satisfies the following condition:
\begin{itemize}
    \item[(x)] If \( K \leq H \), then \( r^{H}_{K}(1_{A(H)}) = 1_{A(K)} \).
\end{itemize}
\end{Def}

\begin{Exam}
Let \( M \) be a \( G \)-algebra over the ring \( R \). Then the fixed-point functor \( \text{FP}_M \) from the previous example is, in fact, a cohomological Green functor for \( G \).

In particular, consider the case when \( M = \text{End}_R(W) \), where \( W \) is an \( R[G] \)-module. In this case, we obtain a Green functor where the multiplication is given by function composition. This Green functor is denoted by \( \text{End}_{R^*}(W) \).
\end{Exam}

\section{${C_{p}}$-Green functors of Lie type}\label{sec3}
In this section, we will define a detailed description of $C_{p}$-Green functors of Lie type along with some examples.
\begin{Def}
A $C_{p}$-Mackey functor $L$ is called a $C_{p}$-Green functors of Lie type if the following conditions hold:
\begin{enumerate}
\item[1.]  $L(C_{p})$ and $L(e)$ both are Lie algebras with respect to some Lie brackets $[,]$.
\item[2.]  The restriction maps $r^{C_{p}}_{e}:L(C_{p})\longrightarrow L(e)$ and conjugation maps $c_{x}:L(e)\longrightarrow L(e)$ are Lie algebra homomorphisms.
\item[3.] Frobenius relations hold: If $e\leq C_{p}$ are subgroups of $C_{p}$, then
$$
[t^{C_{p}}_{e}(a),b]=t^{C_{p}}_{e}[a,r^{C_{p}}_{e}(b)];
$$
$$
[b,t^{C_{p}}_{e}(a)]=t_{e}^{C_{p}}[r^{C^{p}}_{e}(b),a],
$$
for all $a\in L(e)$ and $b\in L(C_{p})$. We express the $C_{p}$-Green functors of Lie type by following diagram:
$$
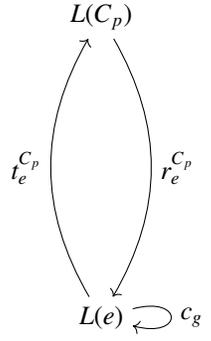
\begin{tikzpicture}[x=3cm,y=2cm]
\node (a) at (0,0) {$L(e)$};
\node (v) at (0,2) {$L(C_{p})$};
\draw[->] (a) to[bend left, "{$t^{C_{p}}_{e}$}"]   (v);
\draw[->] (v) to[bend left,"{$r_{e}^{C_{p}}$}"] (a);
\draw[->] (a) to[loop right,"{$c_{g}$}"] (a);
\end{tikzpicture}
.
$$
We denote the set of all $C_{p}$-Green functors of Lie type by $Green_{Lie}(C_{p})$.
\end{enumerate}
\end{Def}
\begin{Rem}
If \( L(C_p) \) and \( L(e) \) are both graded Lie algebras, and the restriction, transfer, and conjugation maps respect the gradation, then the structure is called a \( C_p \)-Green functor of graded Lie type.
\end{Rem}
\begin{Exam}\label{Heisenberg}
	For $n \ge 1$, let $\mathfrak{h}_{2n+1}$ be the Heisenberg Lie algebra over $\mathbb{K}$ generated by $2n+1$ elements $$\{x_1,\ldots,x_n,y_{1},\ldots,y_{n},z\},$$ where the Lie bracket $[\cdot,\cdot]$ is defined as follows 
	\[
	[x_i,y_{i}] = z,\quad  1 \le i \le n.
	\]   
	The bracket on other combinations of basis elements are zero. For $p\leq n$, let us consider the  following $C_p$ action on the Lie algebra $\mathfrak{h}_{2n+1}$ by Lie algebra automorphisms: 
	 \begin{itemize}
	 \item $\sigma(x_i):=x_{[i+1]},~\sigma(y_i):=y_{[i+1]}$ for $i=1,\ldots, p$,
	 \item $\sigma(x_i):=x_{i},~\sigma(y_i):=y_{i}$, for $i> p$, and $\sigma(z):=z$, 
	 \end{itemize}
	 where $C_p=\langle\sigma\rangle$ and $[i+1]=i+1 (mod~ p)$. Next, take $F(C_p)$ to be the Heisenberg Lie algebra generated by elements $\{\sum_{j=1}^p x_j, x_{p+1},\ldots,x_n,\sum_{j=1}^p y_j, y_{p+1},\ldots,y_{n},z\}$ and $F(e)=\mathfrak{h}_{2n+1}$. Note that $F(C_p)=\mathfrak{h}_{2(n-p)+3}$. Let us consider the following structure maps:
	 \begin{itemize}
	 \item \textbf{Restriction:} The restriction map $r_e^{C_p}:F(C_p)\rightarrow F(e)$ is the inclusion map;
     \item \textbf{Transfer:} The transfer map is defined by 
     \begin{equation*}
     \begin{split}
     t_{e}^{C_p}(x_i)=\sum_{j=1}^p x_j,\quad & t_{e}^{C_p}(y_i)=\sum_{j=1}^p y_j,\quad\text{for } i=1,\ldots,p,\\
        t_{e}^{C_p}(x_i)= p x_i,\quad & t_{e}^{C_p}(y_i)=p y_i,\quad\text{and }\quad t_{e}^{C_p}(z)=p z\quad\text{for } i> p;
     \end{split}
     \end{equation*}
     \item \textbf{Conjugation:} For $\sigma^k\in C_p$ with $k=1,\ldots,p$ the conjugation map $c_{\sigma^k}:\mathfrak{h}_{2n+1}\rightarrow \mathfrak{h}_{2n+1}$ is defined by
     \begin{equation*}
     \begin{split}
     c_{\sigma^k}(x_i)=x_{[i+k]}, \quad& c_{\sigma^k}(y_i)=y_{[i+k]},\quad\text{for } i\leq p;\\
     c_{\sigma^k}(x_i)=x_{i}, \quad& c_{\sigma^k}(y_i)=y_{i},\quad\text{and }\quad c_{\sigma^k}(z)=z,\quad\text{for } i> p.
     \end{split}
     \end{equation*}
     	 \end{itemize}
With the above structure maps, $F$ is a $C_p$-green functor of Lie type.       	 
\end{Exam} 

\begin{Exam}\label{special linear matrices}
	Let us consider the Lie algebra $\mathfrak{sl}_n{\mathbb{C}}$, that is the vector space of $n\times n$ complex matrices with trace $0$ equipped with the commutator bracket. Let us consider the  following $C_2$ action on the Lie algebra $\mathfrak{sl}_n{\mathbb{C}}$ by Lie algebra automorphisms: 
	 $$\sigma (X)=-X^{T}.$$
	 where $X^{T}$ is the transpose of the matrix $X$ and $C_2=\langle\sigma\rangle$. We define $F(C_2):=\mathfrak{so}_n(\mathbb{C})$, i.e., the  Lie algebra of complex skew-symmetric matrices with trace zero, and $F(e):=\mathfrak{sl}_n{\mathbb{C}}$. Then, we consider the following maps
	 \begin{itemize}
	 \item \textbf{Restriction:} The map $r_e^{C_2}:F(C_2)\rightarrow F(e)$ is the inclusion map;
     \item \textbf{Transfer:} The transfer map is defined by 
     \begin{equation*}
     \begin{split}
     t_{e}^{C_2}(X)=X-X^T,\quad \text{for any } X\in \mathfrak{sl}_n(\mathbb{C});
     \end{split}
     \end{equation*}
     \item \textbf{Conjugation:} The conjugation map $c_{\sigma}:\mathfrak{sl}_n(\mathbb{C})\rightarrow \mathfrak{sl}_n(\mathbb{C})$ is defined by
     \begin{equation*}
     c_{\sigma}(X)= -X^T,\quad \text{for any } X\in \mathfrak{sl}_n(\mathbb{C}).
     \end{equation*}
     	 \end{itemize}
With the above structure maps, $F$ is a $C_2$-green functor of Lie type.       	 
\end{Exam} 

\begin{Rem}
The two examples above are specific cases of the following more general example of a Green functor of Lie type, constructed via a fixed-point Mackey functor.
\end{Rem}

\begin{Exam}
Let \( \mathfrak{g} \) be a Lie algebra over a field \( k \), equipped with a \( C_p \)-action by Lie algebra automorphisms. Define a \( C_p \)-Green functor \( F \) of Lie type as follows.

The values of \( F \) are given by:
\[
F(C_p) = \mathfrak{g}^{C_p}, \quad F(e) = \mathfrak{g},
\]
where \( \mathfrak{g}^{C_p} \) denotes the subalgebra of \( C_p \)-fixed points in \( \mathfrak{g} \).

The structure maps of the functor are defined as follows.

\textbf{Restriction:} The map \( r^{C_p}_{e} : \mathfrak{g}^{C_p} \hookrightarrow \mathfrak{g} \) is the natural inclusion.

\textbf{Transfer:} The map \( t^{C_p}_{e} : \mathfrak{g} \to \mathfrak{g}^{C_p} \) is defined by
\[
t^{C_p}_{e}(x) = \sum_{g \in C_p} g \cdot x.
\]

\textbf{Conjugation:} For each \( g \in C_p \), the conjugation map \( c_g : \mathfrak{g} \to \mathfrak{g} \) is given by
\[
c_g(x) = g \cdot x.
\]

Note that both \( \mathfrak{g}^{C_p} \) and \( \mathfrak{g} \) are Lie algebras with their usual Lie brackets. It is straightforward to verify that these data define a \( C_p \)-Green functor of Lie type.
\end{Exam}

\begin{Exam}
Let \( \mathfrak{g} \) be a finite-dimensional Lie algebra over a field \( k \). We define a \( C_p \)-Green functor \( L \) of Lie type as follows.

At the group \( C_p \), define \( L(C_p) = \mathfrak{g} \), equipped with its original Lie bracket \( [-,-]_{\mathfrak{g}} \).

At the trivial group \( e \), define \( L(e) = \mathfrak{g}^{\oplus p} \), the direct sum of \( p \) copies of \( \mathfrak{g} \), with the Lie bracket defined component-wise:
\[
[(x_1, \dots, x_p), (y_1, \dots, y_p)] = ([x_1, y_1]_{\mathfrak{g}}, \dots, [x_p, y_p]_{\mathfrak{g}}).
\]

The structure maps of the functor are as follows.

\textbf{Restriction:} The map \( r^{C_p}_{e} : \mathfrak{g} \to \mathfrak{g}^{\oplus p} \) is the diagonal embedding:
\[
r^{C_p}_{e}(x) = (x, x, \dots, x).
\]

\textbf{Conjugation:} For a generator \( g \in C_p \), the conjugation map \( c_{g} : \mathfrak{g}^{\oplus p} \to \mathfrak{g}^{\oplus p} \) is the cyclic permutation:
\[
c_{g}(x_1, x_2, \dots, x_p) = (x_p, x_1, \dots, x_{p-1}).
\]

\textbf{Transfer:} The map \( t^{C_p}_{e} : \mathfrak{g}^{\oplus p} \to \mathfrak{g} \) is given by component-wise summation:
\[
t^{C_p}_{e}(x_1, \dots, x_p) = x_1 + x_2 + \cdots + x_p.
\]
It is a routine work to check that it is a \(C_{p}\)-Green functor of Lie type.
\end{Exam}
\begin{Exam}
Let $A$ be an associative $R$-algebra equipped with an action of $C_{p}$; equivalently, $A$ is an associative algebra over the group ring $R[C_{p}]$.

$R$-linear derivations, $Der_{R}(A)$, are linear maps $D:A\rightarrow A$ satisfying:
$$
D(ab)=D(a)b+aD(b).
$$
$R[C_{p}]$-linear(invariant) derivations, $Der_{R[C_{p}]}(A)$, are those derivations that commute with the $C_{p}$ action:
$$
D(g\cdot a)=g\cdot D(a), \text{ for all}~g\in G,~a\in A.
$$
It is trivial to show that $Der_{R}(A)$ and $Der_{R[C_{p}]}(A)$ both are lie algebra over the commutative ring $R$ under the commutator bracket and also easy to see that $Der_{R[e]}(A)=Der_{R}(A)$.

Now, we define a $C_{p}$-Green functor $L$ of Lie type as follows.

For the group $C_{p}$, define $L(C_{p})=Der_{R[C_{p}]}(A)$, and for the trivial group $e$, define $L(e)=Der_{R}(A)$.

The family of structure maps are as follows:

\textbf{Restriction:} The restriction map $r^{C_{p}}_{e}:Der_{R[C_{p}]}(A)\rightarrow Der_{R}(A)$ is inclusion map.

\textbf{Transfer:} The transfer map $t^{C_{p}}_{e}:Der_{R}(A)\rightarrow Der_{R[C_{p}]}(A)$ is define as
$$
t^{C_{p}}_{e}(D)(a)=\sum_{g \in C_p}g\cdot D(g^{-1}\cdot a),~~~\text{for all}~a\in A.
$$

\textbf{Conjugation:} For all $g\in C_{p}$, the conjugation map $C_{g}:Der_{R}(A)\rightarrow Der_{R}(A)$ is define as:
$$
c_{g}(D)(a)=g\cdot D(g^{-1}\cdot a),~~~\text{for all }~a\in A.
$$
\end{Exam}

\begin{Def}\label{Def1}
Let $L_{1}, L_{2} \in \mathrm{Green}_{\mathrm{Lie}}(C_{p})$. We define a morphism $f: L_{1} \longrightarrow L_{2}$ in $\mathrm{Green}_{\mathrm{Lie}}(C_{p})$ as a pair $(f_{C_{p}}, f_{e})$ of Lie algebra homomorphisms which commute with the transfer, restriction, and conjugation maps. That is, $f=(f_{C_{p}},f_{e})$ is a morphism for which the following diagram commutes:
\[
\begin{tikzcd}
L_{1}(C_{p}) \arrow{r}{f_{C_{p}}} \arrow[bend left]{d}{r^{C_{p}}_{e}} & L_{2}(C_{p}) \arrow[bend left]{d}{r^{C_{p}}_{e}} \\
L_{1}(e) \arrow{r}{f_{e}} \arrow[bend left]{u}{t^{C_{p}}_{e}} & L_{2}(e) \arrow[bend left]{u}{t^{C_{p}}_{e}}
\end{tikzcd}
\]

In other words, given $L_{1}, L_{2} \in \mathrm{Green}_{\mathrm{Lie}}(C_{p})$ and $x \in C_{p}$, a morphism $f: L_{1} \longrightarrow L_{2}$ consists of Lie algebra homomorphisms
\[
f_{C_{p}}: L_{1}(C_{p}) \longrightarrow L_{2}(C_{p}) \quad \text{and} \quad f_{e}: L_{1}(e) \longrightarrow L_{2}(e)
\]
such that the following diagrams commute:

\[
\begin{tikzcd}
L_{1}(C_{p}) \arrow{r}{f_{C_{p}}} & L_{2}(C_{p}) \\
L_{1}(e) \arrow{r}{f_{e}} \arrow[swap]{u}{t^{C_{p}}_{e}} & L_{2}(e) \arrow[swap]{u}{t^{C_{p}}_{e}}
\end{tikzcd}
\quad;\quad
\begin{tikzcd}
L_{1}(C_{p}) \arrow{r}{f_{C_{p}}} \arrow[swap]{d}{r^{C_{p}}_{e}} & L_{2}(C_{p}) \arrow[swap]{d}{r^{C_{p}}_{e}} \\
L_{1}(e) \arrow{r}{f_{e}} & L_{2}(e)
\end{tikzcd}
\quad;\quad
\begin{tikzcd}
L_{1}(e) \arrow{r}{f_{e}} \arrow[swap]{d}{c_{x, C_{p}}} & L_{2}(e) \arrow[swap]{d}{c_{x, C_{p}}} \\
L_{1}(e) = L_{1}(x_{e}) \arrow{r}{f_{e}} & L_{2}(x_{e}) = L_{2}(e)
\end{tikzcd}
.
\]
\end{Def}

\begin{Def}
Let $M$ be a $C_{p}$-Mackey functor and $L$ a $C_{p}$-Green functor of Lie type. Then $M$ is called a left module over $L$ if the following axioms are satisfied:
\begin{enumerate}
    \item The components $M(e)$ and $M(C_{p})$ are left modules over Lie algebras $L(e)$ and $L(C_{p})$, respectively.
    
    \item For all $a \in L(C_{p})$, $m \in M(C_{p})$, $x \in L(e)$, and $u \in M(e)$, the following compatibility conditions hold:
    \[
    r^{C_{p}}_{e}(a \cdot m) = r^{C_{p}}_{e}(a) \cdot r^{C_{p}}_{e}(m),
    \qquad
    c_{g}(x \cdot u) = c_{g}(x) \cdot c_{g}(u).
    \]
    
    \item \textbf{Frobenius relations}: For all $x \in L(e)$ and $m \in M(C_{p})$,
    \[
    t_{e}^{C_{p}}(x) \cdot m = t_{e}^{C_{p}}(x \cdot r_{e}^{C_{p}}(m)),
    \]
    and for all $a \in L(C_{p})$ and $u \in M(e)$,
    \[
    a \cdot t_{e}^{C_{p}}(u) = t_{e}^{C_{p}}(r_{e}^{C_{p}}(a) \cdot u).
    \]
\end{enumerate}
\end{Def}
Similarly, one can define the notion of a right module for a \( C_p \)-Green functor of Lie type. A module that is both a left and a right module is called a \emph{bimodule}.

\section{Tensor product of ${C_{p}}$-Green functors of Lie type}\label{sec4}
In this section, we define tensor product of ${C_{p}}$-Green functors of Lie type with the help of tensor product of Lie algebras \cite{E91}.
\begin{Def}
Let $L_{1}, L_{2} \in \mathrm{Green}_{\mathrm{Lie}}(C_{p})$. Then, we have the following diagrams:
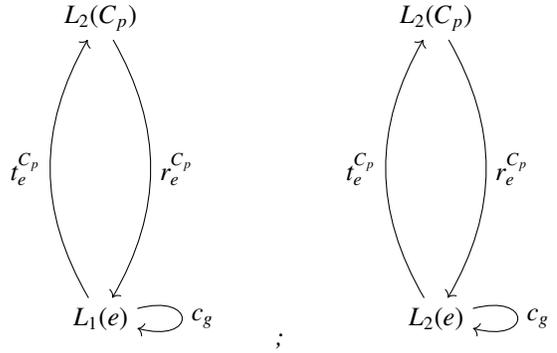
\[
\begin{tikzpicture}[x=3cm,y=2cm]
\node (a) at (0,0) {$L_{1}(e)$};
\node (v) at (0,2) {$L_{2}(C_{p})$};
\draw[->] (a) to[bend left, "{$t^{C_{p}}_{e}$}"] (v);
\draw[->] (v) to[bend left,"{$r_{e}^{C_{p}}$}"] (a);
\draw[->] (a) to[loop right,"{$c_{g}$}"] (a);
\end{tikzpicture}
\qquad\text{;}\qquad
\begin{tikzpicture}[x=3cm,y=2cm]
\node (a) at (0,0) {$L_{2}(e)$};
\node (v) at (0,2) {$L_{2}(C_{p})$};
\draw[->] (a) to[bend left, "{$t^{C_{p}}_{e}$}"] (v);
\draw[->] (v) to[bend left,"{$r_{e}^{C_{p}}$}"] (a);
\draw[->] (a) to[loop right,"{$c_{g}$}"] (a);
\end{tikzpicture}
.
\]

We define the tensor product $L_{1} \otimes L_{2}$ via the following structure:
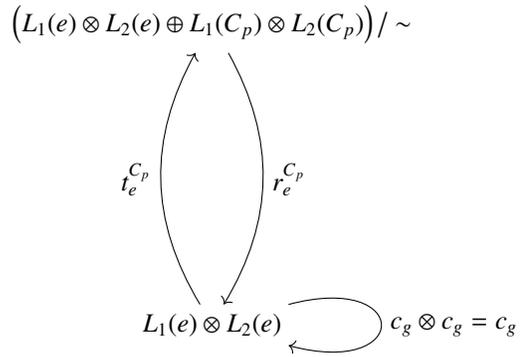
\[
\begin{tikzpicture}[x=3cm,y=2cm]
\node (a) at (0,0) {$L_{1}(e)\otimes L_{2}(e)$};
\node (v) at (0,2) {$\left(L_{1}(e)\otimes L_{2}(e) \oplus L_{1}(C_{p})\otimes L_{2}(C_{p})\right) \big/ \sim$};
\draw[->] (a) to[bend left, "{$t^{C_{p}}_{e}$}"] (v);
\draw[->] (v) to[bend left,"{$r_{e}^{C_{p}}$}"] (a);
\draw[->] (a) to[loop right,"{$c_{g} \otimes c_{g} = c_{g}$}"] (a);
\end{tikzpicture}
.
\]

\begin{enumerate}
    \item The equivalence relation $\sim$ is generated by:
    \[
    a \otimes t^{C_{p}}_{e}(y) \sim r^{C_{p}}_{e}(a) \otimes y,\quad
    t^{C_{p}}_{e}(x) \otimes b \sim x \otimes r^{C_{p}}_{e}(b),
    \]
  
    for all $a \in L_{1}(C_{p})$, $b \in L_{2}(C_{p})$, $x \in L_{1}(e)$, and $y \in L_{2}(e)$.

    \item The conjugation action is given by:
    \[
    c_{g}(x \otimes y) = c_{g}(x) \otimes c_{g}(y),
    \quad \text{and} \quad c_{g}^{p} = \mathrm{id}.
    \]

    \item Elements in the quotient are denoted by angle brackets, e.g., $\langle a \otimes b \rangle$ and $\langle x \otimes y \rangle$, where $a \otimes b \in L_{1}(C_{p}) \otimes L_{2}(C_{p})$ and $x \otimes y \in L_{1}(e) \otimes L_{2}(e)$.

    \item The restriction map $r^{C_{p}}_{e}$ is a homomorphism:
    \[
    r^{C_{p}}_{e}(\langle a \otimes b \rangle) = r^{C_{p}}_{e}(a) \otimes r^{C_{p}}_{e}(b).
    \]

     For $x \otimes y \in L_{1}(e) \otimes L_{2}(e)$, define:
    \[
    r^{C_{p}}_{e}(\langle x \otimes y \rangle) = x \otimes y + c_{g}(x) \otimes c_{g}(y) + \cdots + c_{g}^{p-1}(x) \otimes c_{g}^{p-1}(y).
    \]

    \item The transfer map is defined by:
    \[
    t^{C_{p}}_{e}(x \otimes y) = \langle x \otimes y \rangle.
    \]
\end{enumerate}
\end{Def}

\begin{Rem}
The tensor product $L_{1} \otimes L_{2}$, together with the structure maps of conjugation, restriction, and transfer, forms a Mackey functor. 
\end{Rem}

\subsection{Exterior product of ${C_{p}}$-Green functor of Lie type:}
In this section, we define the exterior product. For all subgroups $H \leq C_{p}$, let $(\overline{L_{1} \otimes L_{2}})(H)$ denote the ideal of $(L_{1} \otimes L_{2})(H)$ generated by all elements of the form:
\[
a_{i} \otimes b_{i} \oplus x_{i} \otimes y_{i} \quad \text{and} \quad a_{i} \otimes b_{i},
\]
where $a_{i}\in L_{1}(e)$, $b_{i}\in L_{2}(e)$,$~x_{i}\in L_{1}(C_{p})$ and $y_{i}\in L_{2}(C_{p})$, where $a_{i}=b_{i}$ and $x_{i}=y_{i}$.

\begin{Def}
Let $L_{1}, L_{2} \in \mathrm{Green}_{\mathrm{Lie}}(C_{p})$. We define the exterior product $L_{1} \wedge L_{2}$ using the following diagram:
\[
\begin{tikzpicture}[x=3cm,y=2cm]
\node (a) at (0,0) {$(L_{1} \otimes L_{2})(e) / (\overline{L_{1} \otimes L_{2}})(e)$};
\node (v) at (0,2) {$(L_{1} \wedge L_{2})(C_{p}) = (L_{1} \otimes L_{2})(C_{p}) / (\overline{L_{1} \otimes L_{2}})(C_{p})$};
\draw[->] (a) to[bend left, "{$t^{C_{p}}_{e}$}"] (v);
\draw[->] (v) to[bend left,"{$r_{e}^{C_{p}}$}"] (a);
\draw[->] (a) to[loop right,"{$c_{g}$}"] (a);
\end{tikzpicture}
\]

\begin{enumerate}
    \item We denote elements in $(L_{1} \wedge L_{2})(C_{p})$ by $\langle a \wedge b \rangle$ and elements in $(L_{1} \wedge L_{2})(e)$ by $\langle x \wedge y \rangle$.
    
    \item The conjugation action is given by:
    \[
    c_{g}(x \wedge y) = c_{g}(x) \wedge c_{g}(y), \quad \text{for } x \wedge y \in (L_{1} \wedge L_{2})(e).
    \]
    
    \item The restriction map $r^{C_{p}}_{e}$ is a homomorphism, that is:
    \[
    r^{C_{p}}_{e}(\langle a \wedge b \rangle) = r^{C_{p}}_{e}(a) \wedge r^{C_{p}}_{e}(b).
    \]

    The restriction map on elements in $(L_{1} \wedge L_{2})(e)$ is defined as:
    \[
    r_{e}^{C_{p}}(\langle x \wedge y \rangle) = x \wedge y + c_{g,e}(x) \wedge c_{g,e}(y) + \cdots + c_{g,e}^{p-1}(x) \wedge c_{g,e}^{p-1}(y).
    \]

    \item The transfer map is defined by:
    \[
    t^{C_{p}}_{e}(x \wedge y) = \langle x \wedge y \rangle.
    \]
\end{enumerate}
\end{Def}

\begin{Rem}
The exterior product $L_{1} \wedge L_{2}$, together with the structure maps (conjugation, restriction, and transfer), forms a Mackey functor. 
\end{Rem}

\section{Chevalley-Eilenberg cohomology of {$C_{p}$}-Green functor of Lie type} \label{sec5}
In this section, we develop a notion of Chevalley-Eilenberg cohomology for the Green functors of Lie type using Chevalley-Eilenberg cohomology \cite{che, Ho} for Lie algebras.

\subsection{$C_{p}$-HOM$(\wedge^{n} L_{1},L_{2}$)}
Let $L_{1}, L_{2} \in \mathrm{Green}_{\mathrm{Lie}}(C_{p})$. In this section, we define the $C_{p}$-equivariant morphism space
\[
C_{p}\text{-}\mathrm{HOM}(\wedge^{n}L_{1}, L_{2}).
\]

\begin{Def}
Let $L_{1}, L_{2} \in \mathrm{Green}_{\mathrm{Lie}}(C_{p})$. We define the $C_{p}$-equivariant homomorphism space
\[
C_{p}\text{-}\mathrm{HOM}(\wedge^{n}L_{1}, L_{2})
\]
using the following diagram:
\[
\begin{tikzpicture}[x=3cm, y=2cm]
\node (a) at (0,0) {$\mathrm{Hom}(\wedge^{n}L_{1}(e), L_{2}(e))$};
\node (v) at (0,2) {$\mathrm{Hom}_{\mathrm{Green}_{\mathrm{Lie}}(C_{p})}(\wedge^{n}L_{1}, L_{2})$};
\draw[->] (a) to[bend left, "{$T$}"] (v);
\draw[->] (v) to[bend left, "{$R$}"] (a);
\draw[->] (a) to[loop right, "{$c_{g}$}"] (a);
\end{tikzpicture}
\]

such that the following conditions hold:

\begin{enumerate}
    \item The set $\mathrm{Hom}_{\mathrm{Green}_{\mathrm{Lie}}(C_{p})}(\wedge^{n}L_{1}, L_{2})$ consists of pairs $(f_{C_{p}}, f_{e})$ of alternating $n$-linear morphisms between $C_{p}$-Green functors of Lie type that commute with restriction, and conjugation maps, and also satisfy Frobenius relation with transfer map. In other words, we have
    \[
    \mathrm{Hom}_{\mathrm{Green}_{\mathrm{Lie}}(C_{p})}(\wedge^{n}L_{1}, L_{2}) = \left\{ (f_{C_{p}}, f_{e}) \mid f_{H}: \wedge^{n}L_{1}(H) \to L_{2}(H) \text{ is a alternating $n$-linear morphism for~} H = e, C_{p} \right\}.
    \]

    \item The space $\mathrm{Hom}(\wedge^{n}L_{1}(e), L_{2}(e))$ is the set of alternating $n$-linear morphism from $\wedge^{n}L_{1}(e)$ to $L_{2}(e)$.

    \item The action of $C_{p}$ on $\mathrm{Hom}(\wedge^{n}L_{1}(e), L_{2}(e))$ is defined as follows: Let $\theta\in \mathrm{Hom}(\wedge^{n}L_{1}(e),L_{2}(e))$.
    \[
    ({}^{g}\theta)(y) := g \cdot \theta(g^{-1} \cdot y), \quad \text{for all } y \in \wedge^{n}L_{1}(e).
    \]

    \item The restriction map $R: \mathrm{Hom}_{\mathrm{Green}_{\mathrm{Lie}}(C_{p})}(\wedge^{n}L_{1}, L_{2}) \to \mathrm{Hom}(\wedge^{n}L_{1}(e), L_{2}(e))$ is defined by:
    \[
    R(f_{C_{p}}, f_{e}) = f_{e}.
    \]

    \item The transfer map
    \[
    T: \mathrm{Hom}(\wedge^{n}L_{1}(e), L_{2}(e)) \longrightarrow \mathrm{Hom}_{\mathrm{Green}_{\mathrm{Lie}}(C_{p})}(\wedge^{n}L_{1}, L_{2})
    \]
    is defined as follows. Let $\gamma: \wedge^{n}L_{1}(e) \to L_{2}(e)$ be an alternating $n$-linear morphism. Then $T(\gamma) = (\mu, \delta)$ is defined by the following commutative diagram:
    \[
    \begin{tikzcd}
    \wedge^{n}L_{1}(C_{p}) \arrow{r}{\mu} \arrow[bend left]{d}{r^{C_{p}}_{e}} & L_{2}(C_{p}) \arrow[bend left]{d}{r^{C_{p}}_{e}} \\
    \wedge^{n}L_{1}(e) \arrow{r}{\delta} \arrow[bend left]{u}{t^{C_{p}}_{e}} & L_{2}(e) \arrow[bend left]{u}{t^{C_{p}}_{e}}
    \end{tikzcd}
    \]
    where
    \[
    \mu(a) := t^{C_{p}}_{e} \circ \gamma \circ r^{C_{p}}_{e}(a),
    \quad \text{and} \quad
    \delta(x) := \sum_{i+j=p} c_{g}^{i} \circ \gamma \circ c_{g}^{j}(x).
    \]
    for all $a \in \wedge^{n}L_{1}(C_{p})$ and $x \in \wedge^{n}L_{1}(e)$.
\end{enumerate}
\end{Def}

The definition of Chevalley--Eilenberg cohomology can be extended to \( C_p \)-Green functors of Lie type, as discussed in this section.
\begin{Def}
Let \( L \) be a \( C_p \)-Green functor of Lie type, and let \( M \) be a bimodule over \( L \). Then, the Chevalley--Eilenberg cohomology of \( L \) with coefficients in \( M \), which again forms a \( C_p \)-Mackey functor, is defined as the cohomology of the following diagram of cochain complexes:
\[
\begin{tikzcd}
M(C_p) \arrow{r}{b_{0}'} \arrow[bend left]{d}{r^{{C_p}}_e} & \mathrm{Hom}_{\mathrm{Mack}(C_p)}(L, M) \arrow{r}{b_{1}'} \arrow[bend left]{d}{r} & \mathrm{Hom}_{\mathrm{Mack}(C_p)}(L \wedge L, M) \arrow{r}{b_{2}'} \arrow[bend left]{d}{r} & \cdots \\
M(e) \arrow{r}{b_0} \arrow[bend left]{u}{t^{{C_p}}_e} & \mathrm{Hom}(L(e), M(e)) \arrow{r}{b_1} \arrow[bend left]{u}{t} & \mathrm{Hom}(L(e) \wedge L(e), M(e)) \arrow{r}{b_2} \arrow[bend left]{u}{t} & \cdots
\end{tikzcd}
\]

\noindent
The coboundary maps are defined as follows:

\begin{enumerate}
    \item[1.] The coboundary \( b_n \) in the lower complex is given by:
    \begin{align*}
    b_n(f_e)(x_1, \dots, x_{n+1}) &= \sum_{i=1}^{n+1} (-1)^{i+1} x_i f_e(x_1, \dots, \widehat{x_i}, \dots, x_{n+1}) \\
    &\quad + \sum_{i<j} (-1)^{i+j} f_e([x_i, x_j], x_1, \dots, \widehat{x_i}, \dots, \widehat{x_j}, \dots, x_{n+1}),
    \end{align*}
    where \( \widehat{x_i} \) denotes omission of the \( i \)-th argument.

    \item[2.] The coboundaries \( b_n' \) in the upper complex are given as follows:
    \begin{enumerate}
        \item[(i)] For \( m \in M(C_p) \):
        \[
        b_0'(m)(a) = a \cdot m, \qquad b_0'(m)(x) = x \cdot r^{{C_p}}_e(m).
        \]

        \item[(ii)] For \( f_{C_p} \in \mathrm{Hom}_{\mathrm{Mack}(C_p)}(L, M) \):
        \[
        b_1'(f_{C_p})(a, b) = a f_{C_p}(b) - b f_{C_p}(a) - f_{C_p}([a, b]),
        \]
        and similarly for \( f_e \in \mathrm{Hom}(L(e), M(e)) \):
        \[
        b_1'(f_e)(x, y) = x f_e(y) - y f_e(x) - f_e([x, y]).
        \]

        \item[(iii)] For \( f_{C_p} \in \mathrm{Hom}(L \wedge L, M) \):
        \begin{align*}
        b_2'(f_{C_p})(a, b, c) &= a f_{C_p}(b, c) - b f_{C_p}(a, c) + c f_{C_p}(a, b) \\
        &\quad - f_{C_p}([a, b], c) + f_{C_p}([a, c], b) - f_{C_p}([b, c], a),
        \end{align*}
        and likewise for \( f_e \in \mathrm{Hom}(L(e) \wedge L(e), M(e)) \):
        \begin{align*}
        b_2'(f_e)(x, y, z) &= x f_e(y, z) - y f_e(x, z) + z f_e(x, y) \\
        &\quad - f_e([x, y], z) + f_e([x, z], y) - f_e([y, z], x).
        \end{align*}

        \item[(iv)] In general, for all \( n \geq 1 \):
        \begin{align*}
        b_n'(f_e)(x_1, \dots, x_{n+1}) &= \sum_{i=1}^{n+1} (-1)^{i+1} x_i f_e(x_1, \dots, \widehat{x_i}, \dots, x_{n+1}) \\
        &\quad + \sum_{i<j} (-1)^{i+j} f_e([x_i, x_j], x_1, \dots, \widehat{x_i}, \dots, \widehat{x_j}, \dots, x_{n+1}),
        \end{align*}
        \begin{align*}
        b_n'(f_{C_p})(a_1, \dots, a_{n+1}) &= \sum_{i=1}^{n+1} (-1)^{i+1} a_i f_{C_p}(a_1, \dots, \widehat{a_i}, \dots, a_{n+1}) \\
        &\quad + \sum_{i<j} (-1)^{i+j} f_{C_p}([a_i, a_j], a_1, \dots, \widehat{a_i}, \dots, \widehat{a_j}, \dots, a_{n+1}).
        \end{align*}
    \end{enumerate}
\end{enumerate}

\noindent
Here, \( a, b, c \in L(C_p) \), \( x, y, z \in L(e) \), \( m \in M(C_p) \), \( u \in M(e) \), \( x_1, \dots, x_{n+1} \in \wedge^{n+1} L(e) \), and \( a_1, \dots, a_{n+1} \in \wedge^{n+1} L(C_p) \).
\begin{prop}
$b_{1}'b_{0}'=0$.
\end{prop}
\begin{Proof*}
For $m\in M(C_{p})$, $(a,b)\in \wedge^{2}L(C_{p})$ and $(x,y)\in \wedge^{2}L(e)$
\begin{align*}
b_{1}'b_{0}'(m)(a,b)&=a\cdot b_{0}'(m)(b)-b\cdot b_{0}'(m)(a)-b_{0}'(m)([a,b])\\
&=a\cdot b\cdot m -b\cdot a\cdot m -[a,b]\cdot m\\
&=a\cdot b\cdot m -b\cdot a\cdot m-(a\cdot b\cdot m -b\cdot a\cdot m)\\
&=0.
\end{align*}
Also, we have
\begin{align*}
b_{1}'b_{0}'(m)(x,y)&=x\cdot b_{0}'(m)(y)-y\cdot b_{0}'(m)(x)-b_{0}'(m)([x,y])\\
&=x\cdot y \cdot r^{C_{p}}_{e}(m)-y\cdot x\cdot r^{C_{p}}_{e}(m)-[x,y]r_{e}^{C_{p}}(m)\\
&=x\cdot y \cdot r^{C_{p}}_{e}(m)-y\cdot x\cdot r^{C_{p}}_{e}(m)-(x\cdot y \cdot r^{C_{p}}_{e}(m)-y\cdot x\cdot r^{C_{p}}_{e}(m))\\
&=0.
\end{align*}
\end{Proof*}

\begin{Rem}
It follows from the above Proposition and Chevalley--Eilenberg coboundary maps \cite{che} that \( b_{n+1}' b_n' = 0 \), and \( b_{n+1} b_n = 0 \) for all $n\geq 0$.

\end{Rem}

Hence, the Chevalley--Eilenberg cohomology of $C_{p}$-Green functor of Lie type is again a  $C_p$-Mackey functor and which is defined by
\[
H^n(L, M) = H^n(C^n(L, M)) = 
\begin{cases}
H^n_{C_p}(C^n_{C_p}(L, M)), \\
H^n_e(C^n_e(L, M)),
\end{cases}
\]
where \( C^n(L, M) = C_p\text{-}\mathrm{HOM}(\wedge^n L, M) \). 
\end{Def}

\begin{Rem}\label{Rem1}
From the above context, \(Z^{2}(L,M)\) of a \(C_p\)-Green functor of Lie type is given by a pair \((Z^2_{C_p}(L,M), Z^2_e(L,M))\).

Note that elements of \(Z^2_{C_p}(L,M)\), which describe \(2\)-cocycles of \(L\) with values in \(M\), are themselves pairs \((f_{C_p}, f_e)\), where
\[
f_{C_p} : \wedge^2 L(C_p) \to M(C_p) \quad \text{and} \quad f_e : \wedge^2 L(e) \to M(e)
\]
are alternating bilinear maps satisfying the following conditions:

\begin{enumerate}
\item[$1.$] $af_{C_{p}}(b,c)-bf_{C_{p}}(a,c)+cf_{C_{p}}(a,b)-f_{C_{p}}([a,b],c)+f_{C_{p}}([a,c],b)-f_{C_{p}}([b,c],a)=0.$
\item[$2.$] $af_{e}(b,c)-bf_{e}(a,c)+cf_{e}(a,b)-f_{e}([a,b],c)+f_{e}([a,c],b)-f_{e}([b,c],a)=0.$
\item[$3.$] $c_{M}(f_{e}(x,y))=f_{e}(c_{L}(x),c_{L}(y)).$
\item[$4.$] $r_{M}(f_{C_{p}}(a,b))=f_{e}(r_{L}(a),r_{L}(b)).$
\item[$5.$] $f_{C_{p}}(t_{L}(x),a)=t_{M}(f_{e}(x,r_{L}(a))).$
\item[$6.$] $f_{C_{p}}(a,t_{L}(a))=t_{M}(f_{e}(r_{L}(x),a)).$
\end{enumerate}
for all $a,b\in L(C_{p})$ and $x,y\in L(e)$.
Moreover, \( Z^{2}_{e}(L, M) \) is the usual Chevalley–Eilenberg \(2\)-cocycle of the Lie algebra \(L(e)\) with values in the module \(M(e)\).

\end{Rem}

\section{Cup Product}\label{sec6}

Let \( L \) be a \( C_p \)-Green functor of Lie type, and let \( M \) be an \( L \)-module with having additional structure of a \( C_p \)-Green functor of Lie type. For example, one could take $M=L$ via adjoint representation. Define the cup products (pairings) as follows:
\[
\cup_{C_p} : C^m_{C_p}(L, M) \times C^n_{C_p}(L, M) \longrightarrow C^{m+n}_{C_p}(L, M),
\]
\[
\cup_{e} : C^m_{e}(L, M) \times C^n_{e}(L, M) \longrightarrow C^{m+n}_{e}(L, M).
\]

These products are defined by
\[
((f_{C_p}, f_e) \cup_{C_p} (g_{C_p}, g_e))((x_1, \ldots, x_{m+n}), (x_1', \ldots, x_{m+n}')) = \left( (f_{C_p} \cup_{C_p} g_{C_p})(x_1, \ldots, x_{m+n}),\, (f_e \cup_{C_p} g_e)(x_1', \ldots, x_{m+n}') \right),
\]
and
\[
(f_e \cup_e g_e)(x_1', \ldots, x_{m+n}').
\]

Here, the products are given explicitly by
\begin{align*}
(f_{C_p} \cup_{C_p} g_{C_p})(x_1, \ldots, x_{m+n}) 
&= \sum_{\sigma} \text{sign}(\sigma) \cdot [f_{C_p}(x_{\sigma(1)}, \ldots, x_{\sigma(m)}),\ g_{C_p}(x_{\sigma(m+1)}, \ldots, x_{\sigma(m+n)})]_{M(C_p)}, \\
(f_e \cup_{C_p} g_e)(x_1', \ldots, x_{m+n}') 
&= \sum_{\sigma} \text{sign}(\sigma) \cdot [f_e(x_{\sigma(1)}', \ldots, x_{\sigma(m)}'),\ g_e(x_{\sigma(m+1)}', \ldots, x_{\sigma(m+n)}')]_{M(e)}, \\
(f_e \cup_e g_e)(x_1', \ldots, x_{m+n}') 
&= \sum_{\sigma} \text{sign}(\sigma) \cdot [f_e(x_{\sigma(1)}', \ldots, x_{\sigma(m)}'),\ g_e(x_{\sigma(m+1)}', \ldots, x_{\sigma(m+n)}')]_{M(e)}.
\end{align*}

In each case, the sum is taken over all permutations \( \sigma \) of the set \( \{1, \ldots, m+n\} \) such that 
\[
\sigma(1) < \sigma(2) < \cdots < \sigma(m) \quad \text{and} \quad \sigma(m+1) < \cdots < \sigma(m+n).
\]
This operation is called the \it{cup product}.

Let
\[
b_{m+n}' : C^{m+n}_{C_p}(L, M) \longrightarrow C^{m+n+1}_{C_p}(L, M), \qquad
b_{m+n} : C^{m+n}_{e}(L, M) \longrightarrow C^{m+n+1}_{e}(L, M)
\]
denote the respective coboundary maps. Then, for all \( (f_{C_p}, f_e) \in C^m_{C_p}(L, M) \), \( (g_{C_p}, g_e) \in C^n_{C_p}(L, M) \), and \( f_e \in C^m_{e}(L, M) \), \( g_e \in C^n_{e}(L, M) \), similar to the computation in \cite{NR}, it is easy to verify the following identity:

\begin{equation}
b_{m+n}'(f_{C_p} \cup_{C_p} g_{C_p}) = b_{m}'(f_{C_p}) \cup_{C_p} g_{C_p} + (-1)^m f_{C_p} \cup_{C_p} b_{n}'(g_{C_p});
\end{equation}

\begin{equation}
b_{m+n}'(f_e \cup_{C_p} g_e) = b_{m}(f_e) \cup_{C_p} g_e + (-1)^m f_e \cup_{C_p} b_{n}(g_e);
\end{equation}

\begin{equation}
b_{m+n}(f_e \cup_e g_e) = b_{m}(f_e) \cup_e g_e + (-1)^m f_e \cup_e b_{n}(g_e).
\end{equation}

\begin{prop}
Let \( f = (f_{C_p}, f_e) \in Z^m_{C_p}(L, M) \) and \( g = (g_{C_p}, g_e) \in Z^n_{C_p}(L, M) \). Then their cup product
\[
f \cup_{C_p} g = (f_{C_p}, f_e) \cup_{C_p} (g_{C_p}, g_e) = (f_{C_p} \cup_{C_p} g_{C_p}, f_e \cup_{C_p} g_e)
\]
belongs to \( Z^{m+n}_{C_p}(L, M) \).
\begin{proof}
If \( f=(f_{C_{p}},f_{e}) \in Z^m_{C_p}(L, M) \), then

\begin{align*}
 \sum_{i=1}^{m+1} (-1)^{i+1} x_i' f_e(x_1', \ldots, \hat{x_i'}, \ldots, x_{m+1}') 
\quad + \sum_{i<j} (-1)^{i+j} f_e([x_i', x_j'], x_1', \ldots, \hat{x_i'}, \ldots, \hat{x_j'}, \ldots, x_{m+1}')=0
\end{align*}

\begin{align*}
 \sum_{i=1}^{m+1} (-1)^{i+1} x_i f_{C_p}(x_1, \ldots, \hat{x_i}, \ldots, x_{m+1}) 
\quad + \sum_{i<j} (-1)^{i+j} f_{C_p}([x_i, x_j], x_1, \ldots, \hat{x_i}, \ldots, \hat{x_j}, \ldots, x_{m+1})=0
\end{align*}

Compatibility conditions
\begin{align*}
C_M(f_e(x_1', \ldots, x_m')) &= f_e(C_L(x_1'), \ldots, C_L(x_m')) \\
r_M f_{C_p}(x_1, \ldots, x_m) &= f_e(r_L(x_1), \ldots, r_L(x_m)) \\
f_{C_p}(t_L(x_{1}'), x_2, \ldots, x_m) &= t_M(f_e(x_1', r_L(x_2), \ldots, r_L(x_m))) \\
f_{C_p}(x_1, \ldots, x_{m-1}, t_L(x_m')) &= t_M(f_e(r_L(x_1), \ldots, r_L(x_{m-1}), x_m')).
\end{align*}

Similarly, if \( g=(g_{C_{p}},g_{e}) \in Z^n_{C_p}(L, M) \), we have

\begin{align*}
 \sum_{i=1}^{n+1} (-1)^{i+1} x_i' g_e(x_1', \ldots, \hat{x_i'}, \ldots, x_{n+1}') 
\quad + \sum_{i<j} (-1)^{i+j} g_e([x_i', x_j'], x_1', \ldots, \hat{x_i'}, \ldots, \hat{x_j'}, \ldots, x_{n+1}')=0;
\end{align*}

\begin{align*}
 \sum_{i=1}^{n+1} (-1)^{i+1} x_i g_{C_p}(x_1, \ldots, \hat{x_i}, \ldots, x_{n+1}) \quad + \sum_{i<j} (-1)^{i+j} g_{C_p}([x_i, x_j], x_1, \ldots, \hat{x_i}, \ldots, \hat{x_j}, \ldots, x_{n+1})=0.
\end{align*}
For compatibility conditions, we have
\begin{align*}
C_M(g_e(x_1', \ldots, x_n')) &= g_e(C_L(x_1'), \ldots, C_L(x_n')) \\
r_M g_{C_p}(x_1, \ldots, x_n) &= g_e(r_L(x_1), \ldots, r_L(x_n)) \\
g_{C_p}(t_L(x_1'), x_2, \ldots, x_n) &= t_M(g_e(x_1', r_L(x_2), \ldots, r_L(x_n))) \\
g_{C_p}(x_1, \ldots, x_{n-1}, t_L(x_n')) &= t_M(g_e(r_L(x_1), \ldots, r_L(x_{n-1}), x_n')).
\end{align*}

Now, we compute
\begin{align*}
&r_M(f_{C_p} \cup_{C_p} g_{C_p})(x_1, \ldots, x_{m+n}) \\
&= \sum_\sigma \operatorname{sgn}(\sigma) r_M [f_{C_p}(x_{\sigma(1)}, \ldots, x_{\sigma(m)}), g_{C_p}(x_{\sigma(m+1)}, \ldots, x_{\sigma(m+n)})]_{M(C_p)} \\
&= \sum_\sigma \operatorname{sgn}(\sigma) [f_e(r_L(x_{\sigma(1)}), \ldots, r_L(x_{\sigma(m)})), g_e(r_L(x_{\sigma(m+1)}), \ldots, r_L(x_{\sigma(m+n)}))]_{M(e)} \\
&= (f_e \cup_e g_e)(r_L(x_1), \ldots, r_L(x_{m+n})).
\end{align*}

Thus,
\[
r_M(f_{C_p} \cup_{C_p} g_{C_p})(x_1, \ldots, x_{m+n}) = (f_e \cup_e g_e)(r_L(x_1), \ldots, r_L(x_{m+n})).
\]

Similarly,
\[
C_M(f_e \cup_{C_p} g_e)(x_1', \ldots, x_{m+n}') = (f_e \cup_{C_p} g_e)(C_L(x_1'), \ldots, C_L(x_{m+n}')).
\]

Now,
\begin{align*}
&t_M((f_e \cup_{C_p} g_e)(x_1', r_L(x_2), \ldots, r_L(x_{m+n}))) \\
&= \sum_\sigma \operatorname{sgn}(\sigma) t_M[f_e(x_{\sigma(1)}', r_L(x_{\sigma(2)}), \ldots, r_L(x_{\sigma(m)})), g_e(r_L(x_{\sigma(m+1)}), \ldots)]_{M(e)} \\
&= \sum_\sigma \operatorname{sgn}(\sigma)[f_{C_p}(t_L(x_{\sigma(1)}'), x_{\sigma(2)}, \ldots, x_{\sigma(m)}), g_{C_p}(x_{\sigma(m+1)}, \ldots)]_{M(C_p)} \\
&= (f_{C_p} \cup_{C_p} g_{C_p})(t_L(x_1'), x_2, \ldots, x_{m+n}).
\end{align*}

Hence,
\[
t_M((f_e \cup_{C_p} g_e)(x_1', r_L(x_2), \ldots, r_L(x_{m+n}))) = (f_{C_p} \cup_{C_p} g_{C_p})(t_L(x_1'), x_2, \ldots, x_{m+n}).
\]

Similarly,
\[
t_M((f_e \cup_{C_p} g_e)(r_L(x_1), \ldots, r_L(x_{m-1}), x_{m+n}')) = (f_{C_p} \cup_{C_p} g_{C_p})(x_1, \ldots, x_{m-1}, t_L(x_{m+n}')).
\]

Therefore, if \( f \in Z^m_{C_p}(L, M) \) and \( g \in Z^n_{C_p}(L, M) \), then by equations (3), (4), (5), and the above relations, \( f \cup_{C_p} g \) satisfies all the conditions for an \((m+n)\)-cocycle. That is,
\[
f \cup_{C_p} g = (f_{C_p}, f_e) \cup_{C_p} (g_{C_p}, g_e) = (f_{C_p} \cup_{C_p} g_{C_p}, f_e \cup_{C_p} g_e) \in Z^{m+n}_{C_p}(L, M).
\]
\end{proof}
\end{prop}

\begin{prop}
If furthermore either \( f \in B^{m}_{C_{p}}(L,M) \) or \( g \in B^{n}_{C_{p}}(L,M) \), then \( f \cup_{C_{p}} g \in B^{m+n}_{C_{p}}(L,M) \).
\end{prop}

\begin{proof}
The proof follows similarly to the preceding proposition.
\end{proof}

\begin{Rem}
If \( f \in Z^{m}_{K}(L,M) \) and \( g \in Z^{n}_{K}(L,M) \) for \( K \leq C_{P} \), then \( f \cup_{K} g \in Z^{m+n}_{K}(L,M) \). Furthermore, if either \( f \in B^{m}_{K}(L,M) \) or \( g \in B^{n}_{K}(L,M) \), then \( f \cup_{K} g \in B^{m+n}_{K}(L,M) \). It follows that there is an induced cup product structure
\[
\cup_{K}:H^{m}_{K}(L,M) \times H^{n}_{K}(L,M) \to H^{m+n}_{K}(L,M),
\]
making the direct sum \( H^{*}_{K}(L,M) = \bigoplus_{n} H^{n}_{K}(L,M) \) a graded Lie algebra for all \( K \leq C_{p} \).
\end{Rem}

\begin{prop}
The cup product also satisfies all pairing conditions.
\end{prop}

\begin{proof}
Let \( f = (f_{C_{p}}, f_{e}) \in C^{m}_{C_{P}}(L,M) \), where \( f_{e} \in C^{m}_{e}(L,M) \), and \( g = (g_{C_{p}}, g_{e}) \in C^{n}_{C_{P}}(L,M) \), where \( g_{e} \in C^{n}_{e}(L,M) \). Consider the morphisms
\[
\begin{aligned}
R_{m} &: C^{m}_{C_{p}}(L,M) \to C^{m}_{e}(L,M), \quad & T_{m} &: C^{m}_{e}(L,M) \to C^{m}_{C_{p}}(L,M), \quad &
\overline{C}^{m}&: C^{m}_{e}(L,M) \to C^{m}_{e}(L,M).
\end{aligned}
\]

We illustrate the naturality using the diagram below

\[
\begin{tikzcd}
	C^{m}_{C_{p}}(L,M) & C^{n}_{C_{p}}(L,M) &&& C^{m+n}_{C_{p}}(L,M) \\
	\\
	C^{m}_{e}(L,M) & C^{n}_{e}(L,M) &&& C^{m+n}_{e}(L,M)
	\arrow["\times"{description}, draw=none, from=1-1, to=1-2]
	\arrow["R_{m}"', curve={height=18pt}, from=1-1, to=3-1]
	\arrow["\cup_{C_{P}}", from=1-2, to=1-5]
	\arrow["R_{n}", curve={height=18pt}, from=1-2, to=3-2]
	\arrow["R_{m+n}"', curve={height=18pt}, from=1-5, to=3-5]
	\arrow["T_{m}", curve={height=18pt}, from=3-1, to=1-1]
	\arrow["\overline{C}^{m}", from=3-1, to=3-1, loop, in=235, out=305, distance=10mm]
	\arrow["\times"{description}, draw=none, from=3-1, to=3-2]
	\arrow["T_{n}"', curve={height=18pt}, from=3-2, to=1-2]
	\arrow["\overline{C}^{n}"', from=3-2, to=3-2, loop, in=305, out=235, distance=10mm]
	\arrow["\cup_{e}"', from=3-2, to=3-5]
	\arrow["T_{m+n}", curve={height=18pt}, from=3-5, to=1-5]
	\arrow["\overline{C}^{m+n}"', from=3-5, to=3-5, loop, in=305, out=235, distance=10mm]
\end{tikzcd}
\]

From the definition of the restriction map \( R_{m+n} \), we get
\[
R_{m+n}((f_{C_{p}}, f_{e}) \cup_{C_{p}} (g_{C_{p}}, g_{e})) = R_{m+n}(f_{C_{P}} \cup_{C_{p}} g_{C_{p}}, f_{e} \cup_{e} g_{e}) = f_{e} \cup_{e} g_{e}.
\]

Also, from the definitions of \( R_m \) and \( R_n \), we have
\[
(R_{m}(f_{C_{p}}, f_{e})) \cup_{e} (R_{n}(g_{C_{P}}, g_{e})) = f_{e} \cup_{e} g_{e}.
\]

Thus,
\[
R_{m+n}((f_{C_{p}}, f_{e}) \cup_{C_{p}} (g_{C_{p}}, g_{e})) = (R_{m}(f_{C_{P}}, f_{e})) \cup_{e} (R_{n}(g_{C_{P}}, g_{e})).
\]

Now, using the definition of the conjugation map \( \overline{C}^{m+n}\), we have
\[
\overline{C}^{m+n}(f_{e} \cup_{e} g_{e})(x_1', \dots, x_{m+n}') = C^{M}(f_{e} \cup_{e} g_{e}) C^{m+n}(x_1', \dots, x^{'}_{m+n}).
\]

Similarly, applying \( \overline{C}^{m} \) and \( \overline{C}^{n} \),
\begin{align*}
&\overline{C}^{m}(f_{e})\cup_{e}\overline{C}^{n}(g_{e})(x_{1}^{'},\cdots, x^{'}_{m+n})\\
&=(C^{M}f_{e}C^{m}\cup_{e}C^{M}g_{e}C^{n})(x_{1}^{'},\cdots, x^{'}_{m+n})\\
&=\sum_{\sigma}sgn(\sigma)[C^{M}f_{e}C^{m}(x_{\sigma(1)},\cdots,x_{\sigma(m)}),C^{M}g_{e}C^{n}(x_{\sigma(m+1)},\cdots,x_{\sigma(m+n)})]_{M(e)}\\
&=\sum_{\sigma}sgn(\sigma)C^{M}[f_{e}(C(x_{\sigma(1)}),\cdots,C(x_{\sigma(m)})),g_{e}(C(x_{\sigma(m+n)}),\cdots,C(x_{\sigma(m+n)}))]_{M(e)}\\
&=C^{M}\sum_{\sigma}sgn(\sigma)[f_{e}(C(x_{\sigma(1)}),\cdots,C(x_{\sigma(m)})),g_{e}(C(x_{\sigma(m+n)}),\cdots,C(x_{\sigma(m+n)}))]_{M(e)}\\
&=C^{M}(f_{e}\cup_{e}g_{e})(C(x_{1}),\cdots,C(x_{m}),C(x_{m+1}),\cdots,C(x_{m+n}))\\
&=C^{M}(f_{e}\cup_{e}g_{e})C^{m+n}(x^{'}_{1},\cdots,x^{'}_{m+n}).
\end{align*}

 We arrive at
\[
\overline{C}^{m}(f_{e}) \cup_{e} \overline{C}^{n}(g_{e}) = \overline{C}^{m+n}(f_{e} \cup_{e} g_{e}),
\]
showing compatibility with conjugation.

Next, using the definition of the transfer map \( T_m \), we have
\begin{align*}
&T_{m}(f_{e})\cup_{C_{p}}(g_{C_{p}},g_{e})\\
&=(t_{M}f_{e}r_{L}^{m},\sum_{i+j=p}C^{i}_{M}f_{e}C^{j}_{L})\cup_{C_{p}}(g_{C_{p}},g_{e})\\
&=(t_{M}f_{e}r_{L}^{m}\cup_{C_{p}}g_{C_{p}},\sum_{i+j=p}C^{i}_{M}f_{e}C^{j}_{L}\cup_{e}g_{e}).
\end{align*}
Now for $(x_{1},\cdots,x_{m+n})\in \wedge^{m+n}L(C_{p})$, we have
\begin{align*}
&(t_{M}f_{e}r_{L}^{m}\cup_{C_{p}}g_{C_{p}})(x_{1},\cdots,x_{m+n})\\
&=\sum_{\sigma}sgn(\sigma)[t_{M}f_{e}r_{L}^{m}(x_{\sigma(1)},\cdots,x_{\sigma(m)}),g_{C_{p}}(x_{\sigma(m+1)},\cdots,x_{\sigma(m+n)})]_{C_{p}}
\end{align*}
For $(x_{1}^{'},\cdots,x_{m+n}^{'})\in \wedge^{m+n}L(e),$ we have
\begin{align*}
&(\sum_{i+j=p}C^{i}_{M}f_{e}C^{j}_{L}\cup_{e}g_{e})(x_{1}^{'},\cdots,x_{m+n}^{'})\\
&=\sum_{\sigma}sgn(\sigma)[\sum_{i+j=P}C^{i}_{M}f_{e}C^{j}_{L}(x^{'}_{\sigma(1)},
\cdots,x^{'}_{\sigma(m)}),g_{e}(x_{\sigma(m+1)}^{'}\cdots,x^{'}_{\sigma(m+n)})]_{M(e)}\\
&=\sum_{\sigma}sgn(\sigma)\sum_{i+j=P}[C^{i}_{M}f_{e}C^{j}_{L}(x^{'}_{\sigma(1)},
\cdots,x^{'}_{\sigma(m)}),g_{e}(x_{\sigma(m+1)}^{'}\cdots,x^{'}_{\sigma(m+n)})]_{M(e)}\\
&=\sum_{i+j=p}\sum_{\sigma}sgn(\sigma)[C^{i}_{M}f_{e}C^{j}_{L}(x^{'}_{\sigma(1)},
\cdots,x^{'}_{\sigma(m)}),g_{e}(x_{\sigma(m+1)}^{'}\cdots,x^{'}_{\sigma(m+n)})]_{M(e)}.
\end{align*}
The definition of transfer map $T_{m+n}$ and restriction map $R_{n}$, we get:
\begin{align*}
&T_{m+n}(f_{e}\cup_{e}R_{n}(g_{C_{p}},g_{e}))\\
&=T_{m+n}(f_{e}\cup_{e}g_{e})\\
&=(t_{M}(f_{e}\cup_{e}g_{e})r_{L}^{m+n},\sum_{i+j=p}C^{i}_{M}(f_{e}\cup_{e}g_{e})C^{j}_{L})
\end{align*}
Now for $(x_{1},\cdots,x_{m+n})\in \wedge^{m+n}L(C_{p})$, we get:
\begin{align*}
&t_{M}(f_{e}\cup_{e}g_{e})r_{L}^{m+n}(x_{1},\cdots,x_{m+n})\\
&=t_{M}(f_{e}\cup_{e}g_{e})(r_{L}(x_{1}),\cdots,r_{L}(x_{m}),\cdots,r_{L}(x_{m+n}))\\
&=t_{M}(\sum_{\sigma}sgn(\sigma)[f_{e}r_{L}^{m}(x_{\sigma(1)},\cdots,x_{\sigma(m)}),g_{e}r_{L}^{n}(x_{\sigma(m+1)},\cdots,x_{\sigma(m+n)})]_{M(e)})\\
&=\sum_{\sigma}sgn(\sigma)t_{M}[f_{e}r_{L}^{m}(x_{\sigma(1)},\cdots,x_{\sigma(m)}),r_{M}^{n}g_{C_{p}}(x_{\sigma(m+1)},\cdots,x_{\sigma(m+n)})]_{M(e)}\\
&=\sum_{\sigma}sgn(\sigma)[t_{M}f_{e}r_{L}^{m}(x_{\sigma(1)},\cdots,x_{\sigma(m)}),g_{C_{p}}(x_{\sigma(m+1)},\cdots,x_{\sigma(m+n)})]_{M_{C_{p}}}
\end{align*}
Now for $(x_{1}^{'},\cdots,x_{m+n}^{'})\in \wedge^{m+n}L(e),$ we have
\begin{align*}
&\sum_{i+j=p}C^{i}_{M}(f_{e}\cup_{e}g_{e})C^{j}_{L}(x_{1}^{'},\cdots,x_{m+n}^{'})\\
&=\sum_{i+j=p}C^{i}_{M}(\sum_{\sigma}sgn(\sigma)[f_{e}C^{j}_{L}(x_{\sigma(1)}^{'},\cdots,x_{\sigma(m)}^{'}),g_{e}C_{L}^{j}(x_{\sigma(m+1)},\cdots,x_{\sigma(m+n)}^{'})]_{M(e)})\\
&=\sum_{i+j=p}\sum_{\sigma}sgn(\sigma)[C^{i}_{M}f_{e}C_{L}^{j}(x_{\sigma(1)}^{'},\cdots,x^{'}_{\sigma(m)}),C_{M}^{i}g_{e}C^{j}_{L}(x_{\sigma(m+1)}^{'},\cdots,x^{'}_{\sigma(m+n)})]_{M(e)}\\
&=\sum_{i+j=p}\sum_{\sigma}sgn(\sigma)[C^{i}_{M}f_{e}C_{L}^{j}(x_{\sigma(1)}^{'},\cdots,x^{'}_{\sigma(m)}),C_{M}^{i}C^{j}_{M}g_{e}(x_{\sigma(m+1)}^{'},\cdots,x^{'}_{\sigma(m+n)})]_{M(e)}\\
&=\sum_{i+j=p}\sum_{\sigma}sgn(\sigma)[C^{i}_{M}f_{e}C_{L}^{j}(x_{\sigma(1)}^{'},\cdots,x^{'}_{\sigma(m)}),g_{e}(x_{\sigma(m+1)}^{'},\cdots,x^{'}_{\sigma(m+n)})]_{M(e)}.
\end{align*}
We compute
\[
T_{m}(f_{e}) \cup_{C_{p}} (g_{C_{p}}, g_{e}) = T_{m+n}(f_{e} \cup_{e} R_{n}(g_{C_{p}}, g_{e})).
\]
Similarly, one shows that
\[
(f_{C_{p}}, f_{e}) \cup_{C_{p}} T_{n}(g_{e}) = T_{m+n}(R_{m}(f_{C_{P}}, f_{e}) \cup_{e} g_{e}).
\]

Hence, the cup product satisfies all the pairing conditions.
\end{proof}

\begin{Rem}
The induced cup product on the graded cohomology \( H^{*}_{K}(L, M) = \bigoplus_{n} H^{n}_{K}(L, M) \), defined for all \( K \leq C_p \), endows \( H^{*}(L, M) \) with the structure of a \( C_p \)-Green functor of graded Lie type.
\end{Rem}

\section{Classification of Singular Extension of $C_{p}$-Green Functors of Lie type}\label{sec7}
It is a well-known result that the second Chevalley-Eilenberg cohomology classifies singular extensions of Lie algebras \cite{che}. In this section, we obtain an analogous result for \( C_p \)-Green functors of Lie type.

\begin{Def}
Let \( L \) be a \( C_p \)-Green functor of Lie type, and let \( M \) be an \( L \)-bimodule. A \textbf{singular extension} \( E \) of \( L \) by \( M \) is an exact sequence of Mackey functors:
\[
\begin{tikzcd}
E: 0 \arrow{r} & M \arrow{r}{i} & B \arrow{r}{j} & L \arrow{r} & 0
\end{tikzcd}
\]
where \( B \) is a \( C_p \)-Green functor of Lie type, \( j \) is a morphism of \( C_p \)-Green functors of Lie type, and \( i \) is a morphism of \( C_p \)-Mackey functors. Furthermore, for all subgroups of \( C_p \), the following sequences:
\[
\begin{tikzcd}
0 \arrow{r} & M(C_p) \arrow{r}{i_{C_p}} & B(C_p) \arrow{r}{j_{C_p}} & L(C_p) \arrow{r} & 0
\end{tikzcd}
\]
and
\[
\begin{tikzcd}
0 \arrow{r} & M(e) \arrow{r}{i_e} & B(e) \arrow{r}{j_e} & L(e) \arrow{r} & 0
\end{tikzcd}
\]
are singular extensions of Lie algebras: \( L(C_p) \) by \( M(C_p) \), and \( L(e) \) by \( M(e) \), respectively.
\end{Def}

\begin{Def}
A singular extension of a \( C_p \)-Green functor of Lie type
\[
\begin{tikzcd}
E: 0 \arrow{r} & M \arrow{r}{i} & B \arrow{r}{j} & L \arrow{r} & 0
\end{tikzcd}
\]
is called \textbf{\( M \)-split} if for the subgroups \( e \) and \( C_p \) of \( C_p \), there exist Lie algebra homomorphisms
\[
s_{C_p} = s(C_p) : L(C_p) \longrightarrow B(C_p), \quad \text{and} \quad s_e = s(e) : L(e) \longrightarrow B(e)
\]
such that:
\begin{enumerate}
    \item[(i)] \( j_{C_p} \circ s_{C_p} = \mathrm{id}_{L(C_p)} \) and \( j_e \circ s_e = \mathrm{id}_{L(e)} \),
    \item[(ii)] The maps \( s_{C_p} \) and \( s_e \) are compatible with the transfer, restriction, and conjugation maps in the following sense:
    \begin{enumerate}
        \item[(a)] \( r_B \circ s_{C_p} = s_e \circ r_L \),
        \item[(b)] \( t_B \circ s_e = s_{C_p} \circ t_L \),
        \item[(c)] \( c_B \circ s_e = s_e \circ c_L \).
    \end{enumerate}
\end{enumerate}
\end{Def}

\begin{Def}
Let
\[
\begin{tikzcd}
E_f: 0 \arrow{r} & M \arrow{r}{i} & B_f \arrow{r}{j} & L \arrow{r} & 0
\end{tikzcd}
\quad \text{and} \quad
\begin{tikzcd}
E_g: 0 \arrow{r} & M \arrow{r}{i'} & B_g \arrow{r}{j'} & L \arrow{r} & 0
\end{tikzcd}
\]
be two extensions of \( L \) by \( M \). We say that \( E_f \) and \( E_g \) are \textbf{equivalent extensions} if there exists a natural isomorphism \( \beta: B_f \to B_g \) that commutes with the restriction and transfer maps.

In other words, there exists a pair of Lie algebra homomorphisms \( (\beta_{C_p}, \beta_e) \) such that the following diagram commutes:
\[\begin{tikzcd}
	0 & {M(C_{p})} && {B_{f_{C_{p}}}(C_{p})} && {L(C_{p})} && 0 \\
	0 && {M(C_{p})} && {B_{g_{C_{p}}}(C_{p})} && {L(C_{p})} & 0 \\
	0 & {M(e)} && {B_{f_{e}}(e)} && {L(e)} && 0 \\
	0 && {M(e)} && {B_{g_{e}}(e)} && {L(e)} & 0
	\arrow[from=1-1, to=1-2]
	\arrow["i_{C_{p}}", from=1-2, to=1-4]
	\arrow["r_{M}"{description, pos=0.7}, curve={height=-18pt}, from=1-2, to=3-2]
	\arrow["j_{C_{p}}", from=1-4, to=1-6]
	\arrow["\beta_{C_{p}}"{pos=0.6}, from=1-4, to=2-5]
	\arrow["r_{B}"{description, pos=0.7}, curve={height=-18pt}, from=1-4, to=3-4]
	\arrow[from=1-6, to=1-8]
	\arrow[equals, from=1-6, to=2-7]
	\arrow["r_{L}"{description, pos=0.7}, curve={height=-18pt}, from=1-6, to=3-6]
	\arrow[from=2-1, to=2-3]
	\arrow[equals, from=2-3, to=1-2]
	\arrow["{i^{'}_{C_{P}}}", from=2-3, to=2-5]
	\arrow["r_{M}"{description, pos=0.7}, curve={height=-18pt}, from=2-3, to=4-3]
	\arrow["{j^{'}_{C_{p}}}", from=2-5, to=2-7]
	\arrow["r_{B}"{description, pos=0.7}, curve={height=-18pt}, from=2-5, to=4-5]
	\arrow[from=2-7, to=2-8]
	\arrow["r_{L}"{description, pos=0.7}, curve={height=-18pt}, from=2-7, to=4-7]
	\arrow[from=3-1, to=3-2]
	\arrow["t_{M}"{description, pos=0.7}, curve={height=-18pt}, from=3-2, to=1-2]
	\arrow["i_{e}", from=3-2, to=3-4]
	\arrow[equals, from=3-2, to=4-3]
	\arrow["t_{B}"{description, pos=0.7}, curve={height=-18pt}, from=3-4, to=1-4]
	\arrow["j_{e}", from=3-4, to=3-6]
	\arrow["\beta_{e}"'{pos=0.4}, shorten >=6pt, from=3-4, to=4-5]
	\arrow["t_{L}"{description, pos=0.7}, curve={height=-18pt}, from=3-6, to=1-6]
	\arrow[from=3-6, to=3-8]
	\arrow[equals, from=3-6, to=4-7]
	\arrow[from=4-1, to=4-3]
	\arrow["t_{M}"{description, pos=0.7}, curve={height=-18pt}, from=4-3, to=2-3]
	\arrow["{i^{'}_{e}}", from=4-3, to=4-5]
	\arrow["t_{B}"{description, pos=0.7}, curve={height=-18pt}, from=4-5, to=2-5]
	\arrow["{j^{'}_{e}}", from=4-5, to=4-7]
	\arrow["t_{L}"{description, pos=0.7}, curve={height=-18pt}, from=4-7, to=2-7]
	\arrow[from=4-7, to=4-8]
\end{tikzcd}\]
\end{Def}
\begin{notation}
Let \( L \) be a \( C_p \)-Green functor of Lie type, and let \( M \) be an \( L \)-bimodule. We define \( \operatorname{Ext}(L, M) \) to be the set of equivalence classes of \( M \)-split extensions of \( L \) by \( M \).
\end{notation}

\begin{prop}\label{Prob2}
Let \( L \) be a \( C_p \)-Green functor of Lie type, and \( M \) be an \( L \)-bimodule Mackey functor. Consider an \( M \)-split extension of \( L \) by \( M \):
\[
\begin{tikzcd}
0 \arrow{r} & M \arrow{r}{i} & B \arrow{r}{j} & L \arrow[bend left=15]{l}{s} \arrow{r} & 0
\end{tikzcd}
\]
Then, for the subgroups \( e \) and \( C_p \) of \( C_p \), the following identities hold:
\[
i_{C_p}(xu) = [s_{C_p}(x), i_{C_p}(u)]_{B(C_p)} \quad \text{for all } x \in L(C_p),~ u \in M(C_p),
\]
\[
i_e(am) = [s_e(a), i_e(m)]_{B(e)} \quad \text{for all } a \in L(e),~ m \in M(e).
\]
\begin{proof}
Since the extension is \( M \)-split, we have
\[
\begin{tikzcd}
0 \arrow{r} & M(C_p) \arrow{r}{i_{C_p}} & B(C_p) \arrow{r}{j_{C_p}} & L(C_p) \arrow[bend left=15]{l}{s_{C_p}} \arrow{r} & 0,
\end{tikzcd}
\quad
\begin{tikzcd}
0 \arrow{r} & M(e) \arrow{r}{i_e} & B(e) \arrow{r}{j_e} & L(e) \arrow[bend left=15]{l}{s_e} \arrow{r} & 0.
\end{tikzcd}
\]

We verify the identity at \( C_p \). For \( x \in L(C_p) \), \( u \in M(C_p) \), define
\[
v := i_{C_p}(xu) - [s_{C_p}(x), i_{C_p}(u)]_{B(C_p)}.
\]
Applying \( j_{C_p} \), we compute:
\begin{align*}
j_{C_p}(v) &= j_{C_p}(i_{C_p}(xu)) - j_{C_p}([s_{C_p}(x), i_{C_p}(u)]) \\
&= 0 - [j_{C_p}(s_{C_p}(x)), j_{C_p}(i_{C_p}(u))] \\
&= 0 - [x, 0] = 0.
\end{align*}
Thus, \( v \in \ker j_{C_p} = \mathrm{im}\, i_{C_p} \), so \( v = i_{C_p}(w) \) for some \( w \in M(C_p) \). Applying \( j_{C_p} \) again gives \( j_{C_p}(v) = j_{C_p}(i_{C_p}(w)) = 0 \), so \( i_{C_p}(w) = 0 \), and hence \( w = 0 \) (since \( i_{C_p} \) is injective). Therefore, \( v = 0 \), and
\[
i_{C_p}(xu) = [s_{C_p}(x), i_{C_p}(u)]_{B(C_p)}.
\]

A similar argument holds at \( e \),hence,
\[
i_e(am) = [s_e(a), i_e(m)]_{B(e)}.
\]
\end{proof}
\end{prop}

\begin{prop}\label{Prob1}
Let \( (f_{C_p}, f_e) \) be a \( C_p \)-Chevalley–Eilenberg 2-cocycle associated with a \( C_p \)-Green functor of Lie type. Then one can construct a new \( C_p \)-Green functor of Lie type, denoted \( B_{f_{C_p}, f_e} \), defined as follows.

The value at \( C_p \) is the Lie algebra \( B_{f_{C_p}}(C_p) = M(C_p) \oplus L(C_p) \), equipped with the bracket
\[
\langle (u_1, x_1), (u_2, x_2) \rangle = \left( x_1 u_2 - x_2 u_1 + f_{C_p}(x_1, x_2), [x_1, x_2] \right).
\]
Similarly, at the trivial subgroup \( e \), we define \( B_{f_e}(e) = M(e) \oplus L(e) \) with the bracket
\[
\langle (m_3, a_3), (m_4, a_4) \rangle = \left( a_3 m_4 - a_4 m_3 + f_e(a_3, a_4), [a_3, a_4] \right).
\]
The structure maps are defined componentwise:
\[
c_B(m, a) = (c_M(m), c_L(a)), \quad
t_B(m, a) = (t_M(m), t_L(a)), \quad
r_B(u, x) = (r_M(u), r_L(x)).
\]

The relationships among these values and maps are summarized in the diagram:
$$
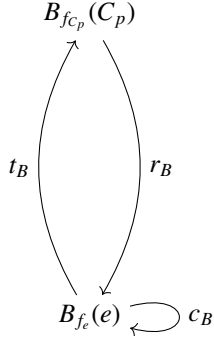
\begin{tikzpicture}[x=3cm,y=2cm]
\node (a) at (0,0) {$B_{f_{e}}(e)$};
\node (v) at (0,2) {$B_{f_{C_{p}}}(C_{p})$};
\draw[->] (a) to[bend left, "{$t_{B}$}"]   (v);
\draw[->] (v) to[bend left,"{$r_{B}$}"] (a);
\draw[->] (a) to[loop right,"{$c_{B}$}"] (a);
\end{tikzpicture}
$$
\end{prop}

\begin{proof}
We must verify that \( B_{f_{C_p}, f_e} \) satisfies the axioms of a \( C_p \)-Green functor of Lie type.

First, note that both \( B_{f_{C_p}}(C_p) \) and \( B_{f_e}(e) \) are Lie algebras by construction, since \( f_{C_p} \) and \( f_e \) are 2-cocycles. We begin by checking that the conjugation map \( c_B \) is a Lie algebra homomorphism. Let \( (m, a), (n, b) \in B_{f_e}(e) \). Then:
\begin{align*}
c_B\big(\langle (m,a), (n,b) \rangle\big)
&= c_B\left(an - bm + f_e(a, b), [a, b]\right) \\
&= \left(c_M(an - bm + f_e(a, b)), c_L([a, b])\right) \\
&= \left(c_L(a)c_M(n) - c_L(b)c_M(m) + c_M(f_e(a, b)), [c_L(a), c_L(b)]\right).
\end{align*}
On the other hand, we compute:
\begin{align*}
\langle c_B(m, a), c_B(n, b) \rangle
&= \langle (c_M(m), c_L(a)), (c_M(n), c_L(b)) \rangle \\
&= \left(c_L(a)c_M(n) - c_L(b)c_M(m) + f_e(c_L(a), c_L(b)), [c_L(a), c_L(b)]\right).
\end{align*}
By condition (3) of remark \eqref{Rem1} (which describes the Chevalley–Eilenberg cocycle compatibility), it follows that these expressions are equal, so \( c_B \) is a Lie algebra homomorphism.

Next, we verify that the restriction map \( r_B \) is also a Lie algebra homomorphism. Let \( (u, x), (v, y) \in B_{f_{C_p}}(C_p) \). Then:
\begin{align*}
r_B\big(\langle (u,x), (v,y) \rangle\big)
&= r_B\left(xv - yu + f_{C_p}(x, y), [x, y]\right) \\
&= \left(r_L(x) r_M(v) - r_L(y) r_M(u) + r_M(f_{C_p}(x, y)), r_L([x, y])\right) \\
&= \left(r_L(x) r_M(v) - r_L(y) r_M(u) + f_e(r_L(x), r_L(y)), [r_L(x), r_L(y)]\right),
\end{align*}
where the final equality uses condition (4) of remark \eqref{Rem1}, which ensures compatibility of cocycles under restriction. Meanwhile:
\begin{align*}
\langle r_B(u, x), r_B(v, y) \rangle
&= \langle (r_M(u), r_L(x)), (r_M(v), r_L(y)) \rangle \\
&= \left(r_L(x) r_M(v) - r_L(y) r_M(u) + f_e(r_L(x), r_L(y)), [r_L(x), r_L(y)]\right).
\end{align*}
Thus, \( r_B \) is also a Lie homomorphism.

To verify the Frobenius relations, let \( (m, a) \in B_{f_e}(e) \) and \( (u, x) \in B_{f_{C_p}}(C_p) \). We compute:
\begin{align*}
\langle t_B(m, a), (u, x) \rangle
&= \langle (t_M(m), t_L(a)), (u, x) \rangle \\
&= \left(t_L(a)u - x t_M(m) + f_{C_p}(t_L(a), x), [t_L(a), x]\right).
\end{align*}
On the other hand:
\begin{align*}
t_B\big(\langle (m, a), r_B(u, x) \rangle\big)
&= t_B\left(a r_M(u) - r_L(x) m + f_e(a, r_L(x)), [a, r_L(x)]\right) \\
&= \left(t_M(a r_M(u) - r_L(x) m + f_e(a, r_L(x))), t_L([a, r_L(x)])\right) \\
&= \left(t_L(a) u - x t_M(m) + f_{C_p}(t_L(a), x), [t_L(a), x]\right),
\end{align*}
where we have again used the compatibility conditions from the definition of module over a \( C_p \)-Green functor of Lie type and condition (5) of remark \eqref{Rem1}.

Finally, it follows analogously that
\[
\langle (u,x), t_B(m,a) \rangle = t_B\big(\langle r_B(u,x), (m,a) \rangle\big),
\]
confirming the adjoint Frobenius identity, as required by condition (6) of remark \eqref{eq5}.

Therefore, all axioms are satisfied, and \( B_{f_{C_p}, f_e} \) is indeed a \( C_p \)-Green functor of Lie type.
\end{proof}

\begin{thm}
Let $L \in \mathrm{Green}_{\mathrm{Lie}}(C_p)$ and let $M$ be an $L$-bimodule. Then there is a one-to-one correspondence between the elements of $\mathrm{Ext}(L, M)$ and the elements of the second cohomology group $H^2_{\mathrm{Green}_{\mathrm{Lie}}(C_p)}(L, M)$.
\end{thm}

\begin{proof}
To prove the theorem, we proceed as follows:

\begin{itemize}
    \item[(a)] First, we construct a well-defined map
    \[
    \mathrm{Ext}(L, M) \longrightarrow H^{2}_{\mathrm{Green}_{\mathrm{Lie}}(C_p)}(L, M).
    \]
    
    \item[(b)] Next, we construct a well-defined map in the reverse direction,
    \[
    H^{2}_{\mathrm{Green}_{\mathrm{Lie}}(C_p)}(L, M) \longrightarrow \mathrm{Ext}(L, M).
    \]
    
    \item[(c)] Finally, we show that these two maps are inverses of each other.
\end{itemize}
We now proceed to prove the theorem step by step:
\begin{enumerate}
\item[(a)]
\end{enumerate}

Consider a singular extension of Green functors of Lie type:
\[
\begin{tikzcd}
E: 0 \arrow{r} & M \arrow{r}{i} & B \arrow{r}{j} & L \arrow{r} & 0
\end{tikzcd}
\]
and let
\[
\begin{tikzcd}
s_{C_p}: L(C_p) \arrow{r} & B(C_p)
\end{tikzcd}
\quad \text{and} \quad
\begin{tikzcd}
s_e: L(e) \arrow{r} & B(e)
\end{tikzcd}
\]
be module homomorphisms such that
\[
j_{C_p} \circ s_{C_p} = \mathrm{id}_{L(C_p)} \quad \text{and} \quad j_e \circ s_e = \mathrm{id}_{L(e)}.
\]

Then for all \( x, y \in L(C_p) \) and \( a, b \in L(e) \), there exist unique alternating bilinear functions \( f_{C_p}(x, y) \in M(C_p) \) and \( f_e(a, b) \in M(e) \) such that
\begin{equation} \label{eq3}
[s_{C_p}(x), s_{C_p}(y)]_B = s_{C_p}[x, y]_L + i_{C_p} f_{C_p}(x, y),
\end{equation}
and
\begin{equation} \label{eq4}
[s_e(a), s_e(b)]_B = s_e[a, b]_L + i_e f_e(a, b).
\end{equation}

For all \( x, y, z \in L(C_p) \), the Jacobiator in \( B \) expands as follows:
\begin{equation} \label{eq5}
\begin{aligned}
[[s_{C_p}(x), s_{C_p}(y)]_B, s_{C_p}(z)]_B 
&= [s_{C_p}[x, y]_L + i_{C_p} f_{C_p}(x, y), s_{C_p}(z)]_B \\
&= [s_{C_p}[x, y]_L, s_{C_p}(z)]_B + [i_{C_p} f_{C_p}(x, y), s_{C_p}(z)]_B \\
&= s_{C_p}[[x, y], z]_L + i_{C_p} f_{C_p}([x, y]_L, z) - [s_{C_p}(z), i_{C_p} f_{C_p}(x, y)]_B.
\end{aligned}
\end{equation}

\begin{equation} \label{eq6}
\begin{aligned}
[[s_{C_p}(y), s_{C_p}(z)]_B, s_{C_p}(x)]_B 
&= [s_{C_p}[y, z]_L + i_{C_p} f_{C_p}(y, z), s_{C_p}(x)]_B \\
&= [s_{C_p}[y, z]_L, s_{C_p}(x)]_B + [i_{C_p} f_{C_p}(y, z), s_{C_p}(x)]_B \\
&= s_{C_p}[[y, z], x]_L + i_{C_p} f_{C_p}([y, z]_L, x) - [s_{C_p}(x), i_{C_p} f_{C_p}(y, z)]_B.
\end{aligned}
\end{equation}

\begin{equation} \label{eq7}
\begin{aligned}
[[s_{C_p}(z), s_{C_p}(x)]_B, s_{C_p}(y)]_B 
&= [s_{C_p}[z, x]_L + i_{C_p} f_{C_p}(z, x), s_{C_p}(y)]_B \\
&= [s_{C_p}[z, x]_L, s_{C_p}(y)]_B + [i_{C_p} f_{C_p}(z, x), s_{C_p}(y)]_B \\
&= s_{C_p}[[z, x], y]_L + i_{C_p} f_{C_p}([z, x]_L, y) - [s_{C_p}(y), i_{C_p} f_{C_p}(z, x)]_B.
\end{aligned}
\end{equation}
Therefore, the multiplication in \( B(C_p) \) satisfies the Jacobi identity. This follows from equations~\eqref{eq5}, \eqref{eq6}, and \eqref{eq7}, which imply that:
\[
i_{C_p} f_{C_p}([x,y]_L, z) - [s_{C_p}(z), i_{C_p} f_{C_p}(x,y)]_B 
+ i_{C_p} f_{C_p}([y,z]_L, x) - [s_{C_p}(x), i_{C_p} f_{C_p}(y,z)]_B 
+ i_{C_p} f_{C_p}([z,x]_L, y) - [s_{C_p}(y), i_{C_p} f_{C_p}(z,x)]_B = 0.
\]

Now, by Proposition~(8.5) and the injectivity of the morphism \( i_{C_p} \), it follows that:
\[
x f_{C_p}(y,z) - y f_{C_p}(x,z) + z f_{C_p}(x,y) 
- f_{C_p}([x,y]_L, z) + f_{C_p}([x,z]_L, y) - f_{C_p}([y,z]_L, x) = 0.
\]

Similarly, for \( a, b, c \in L(e) \), we compute the Jacobiator in \( B(e) \):

\begin{equation} \label{eq8}
\begin{aligned}
[[s_e(a), s_e(b)]_B, s_e(c)]_B 
&= [s_e[a,b]_L + i_e f_e(a,b), s_e(c)]_B \\
&= [s_e[a,b]_L, s_e(c)]_B + [i_e f_e(a,b), s_e(c)]_B \\
&= s_e[[a,b],c]_L + i_e f_e([a,b]_L, c) - [s_e(c), i_e f_e(a,b)]_B.
\end{aligned}
\end{equation}

\begin{equation} \label{eq9}
\begin{aligned}
[[s_e(b), s_e(c)]_B, s_e(a)]_B 
&= [s_e[b,c]_L + i_e f_e(b,c), s_e(a)]_B \\
&= [s_e[b,c]_L, s_e(a)]_B + [i_e f_e(b,c), s_e(a)]_B \\
&= s_e[[b,c],a]_L + i_e f_e([b,c]_L, a) - [s_e(a), i_e f_e(b,c)]_B.
\end{aligned}
\end{equation}

\begin{equation} \label{eq10}
\begin{aligned}
[[s_e(c), s_e(a)]_B, s_e(b)]_B 
&= [s_e[c,a]_L + i_e f_e(c,a), s_e(b)]_B \\
&= [s_e[c,a]_L, s_e(b)]_B + [i_e f_e(c,a), s_e(b)]_B \\
&= s_e[[c,a],b]_L + i_e f_e([c,a]_L, b) - [s_e(b), i_e f_e(c,a)]_B.
\end{aligned}
\end{equation}

Therefore, the multiplication in \( B(e) \) also satisfies the Jacobi identity, as it follows from equations~\eqref{eq8}, \eqref{eq9}, and \eqref{eq10} that:
\[
i_e f_e([a,b]_L, c) - [s_e(c), i_e f_e(a,b)]_B 
+ i_e f_e([b,c]_L, a) - [s_e(a), i_e f_e(b,c)]_B 
+ i_e f_e([c,a]_L, b) - [s_e(b), i_e f_e(c,a)]_B = 0.
\]

Again, applying Proposition~(8.5) and the injectivity of the morphism \( i_e \), we conclude that:
\[
a f_e(b,c) - b f_e(a,c) + c f_e(a,b) 
- f_e([a,b]_L, c) + f_e([a,c]_L, b) - f_e([b,c]_L, a) = 0.
\]
We now show that the pair of bilinear alternating maps \((f_{C_p}, f_e)\) defines a Chevalley–Eilenberg 2-cocycle of $C_{p}$-Green functor of Lie type.

Recall that the conjugation map \(c_B : B(e) \to B(e)\) is a Lie homomorphism. Applying \(c_B\) to equation~\eqref{eq4} yields:
\begin{align*}
c_B[s_e(a), s_e(b)]_B &= c_B(s_e[a,b]_L + i_e f_e(a,b)) \\
[\ c_B(s_e(a)),\ c_B(s_e(b))]_B &= c_B(s_e[a,b]_L) + c_B(i_e f_e(a,b)) \\
[\ s_e(c_L(a)),\ s_e(c_L(b))]_B &= s_e(c_L[a,b]_L) + i_e c_M(f_e(a,b)) \\
s_e[c_L(a), c_L(b)]_L + i_e f_e(c_L(a), c_L(b)) &= s_e[c_L(a), c_L(b)]_L + i_e c_M(f_e(a,b)) \\
f_e(c_L(a), c_L(b)) &= c_M(f_e(a,b)).
\end{align*}

Similarly, the restriction map \(r_B : B(C_p) \to B(e)\) is a Lie homomorphism. Applying \(r_B\) to equation~\eqref{eq3} gives:
\begin{align*}
r_B[s_{C_p}(x), s_{C_p}(y)]_B &= r_B(s_{C_p}[x,y]_L + i_{C_p} f_{C_p}(x,y)) \\
[\ r_B(s_{C_p}(x)),\ r_B(s_{C_p}(y))]_B &= r_B(s_{C_p}[x,y]_L) + r_B(i_{C_p} f_{C_p}(x,y)) \\
[\ s_e(r_L(x)),\ s_e(r_L(y))]_B &= s_e(r_L[x,y]_L) + i_e r_M(f_{C_p}(x,y)) \\
s_e[r_L(x), r_L(y)]_L + i_e f_e(r_L(x), r_L(y)) &= s_e[r_L(x), r_L(y)]_L + i_e r_M(f_{C_p}(x,y)) \\
f_e(r_L(x), r_L(y)) &= r_M(f_{C_p}(x,y)).
\end{align*}

From equation~\eqref{eq3}, we also have:
\begin{align*}
[s_{C_p}(y), s_{C_p}(x)]_B &= s_{C_p}[y,x]_L + i_{C_p} f_{C_p}(y,x), \\
\text{so that} \quad i_{C_p} f_{C_p}(y,x) &= [s_{C_p}(y), s_{C_p}(x)]_B - s_{C_p}[y,x]_L.
\end{align*}
Substituting \(y = t_L(a)\) gives:
\begin{align*}
i_{C_p} f_{C_p}(t_L(a), x) &= [s_{C_p}(t_L(a)), s_{C_p}(x)]_B - s_{C_p}[t_L(a), x]_L \\
&= [t_B(s_e(a)), s_{C_p}(x)]_B - s_{C_p}(t_L[a, r_L(x)]_L) \qquad \text{(Frobenius relation)} \\
&= t_B([s_e(a), r_B(s_{C_p}(x))]_B) - t_B(s_e[a, r_L(x)]_L) \\
&= t_B([s_e(a), s_e(r_L(x))]_B) - t_B(s_e[a, r_L(x)]_L) \\
&= t_B(s_e[a, r_L(x)]_L + i_e f_e(a, r_L(x))) - t_B(s_e[a, r_L(x)]_L) \qquad \text{(by~\eqref{eq4})} \\
&= t_B(i_e f_e(a, r_L(x))) \\
&= i_{C_p} t_M(f_e(a, r_L(x))),
\end{align*}
and hence,
\[
f_{C_p}(t_L(a), x) = t_M(f_e(a, r_L(x))).
\]

A similar calculation applies when substituting \(y = t_L(a)\) on the right. From equation~\eqref{eq3}, we have
\begin{align*}
[s_{C_p}(x), s_{C_p}(y)]_B &= s_{C_p}[x,y]_L + i_{C_p} f_{C_p}(x,y), \\
\text{so that} \quad i_{C_p} f_{C_p}(x,y) &= [s_{C_p}(x), s_{C_p}(y)]_B - s_{C_p}[x,y]_L.
\end{align*}
Substituting \(y = t_L(a)\) yields:
\begin{align*}
i_{C_p} f_{C_p}(x, t_L(a)) &= [s_{C_p}(x), s_{C_p}(t_L(a))]_B - s_{C_p}[x, t_L(a)]_L \\
&= [s_{C_p}(x), t_B(s_e(a))]_B - s_{C_p}(t_L[r_L(x), a]_L) \qquad \text{(Frobenius relation)} \\
&= t_B([r_B(s_{C_p}(x)), s_e(a)]_B) - t_B(s_e[r_L(x), a]_L) \\
&= t_B([s_e(r_L(x)), s_e(a)]_B) - t_B(s_e[r_L(x), a]_L) \\
&= t_B(s_e[r_L(x), a]_L + i_e f_e(r_L(x), a)) - t_B(s_e[r_L(x), a]_L) \qquad \text{(by~\eqref{eq4})} \\
&= t_B(i_e f_e(r_L(x), a)) \\
&= i_{C_p} t_M(f_e(r_L(x), a)),
\end{align*}
and thus,
$$
f_{C_p}(x, t_L(a)) = t_M(f_e(r_L(x), a)).
$$
Hence, $f_{C_{p}}$ and $f_{e}$ satisfy all the conditions in the definition of a Chevalley–Eilenberg 2-cocycle for $C_{p}$-Green functors of Lie type. Let
\[
s_{C_{p}}': L(C_{p}) \longrightarrow B(C_{p})
\quad \text{and} \quad
s_{e}': L(e) \longrightarrow B(e)
\]
be module homomorphisms, and let
\[
g_{C_{p}}: L(C_{p}) \times L(C_{p}) \longrightarrow M(C_{p})
\quad \text{and} \quad
g_{e}: L(e) \times L(e) \longrightarrow M(e)
\]
be the Chevalley–Eilenberg 2-cocycles of $C_{p}$-Green functors of Lie type corresponding to the choices of $s_{C_{p}}'$ and $s_{e}'$. Then,
\[
j_{C_{p}} \circ s_{C_{p}}(x) = x = j_{C_{p}} \circ s_{C_{p}}'(x), \quad
j_{e} \circ s_{e}(a_{1}) = a_{1} = j_{e} \circ s_{e}'(a_{1})
\]
for every $x \in L(C_{p})$ and $a_{1} \in L(e)$. Therefore, there exist maps $h_{C_{p}}: L(C_{p}) \to M(C_{p})$ and $h_{e}: L(e) \to M(e)$ such that
\begin{equation}\label{eq11}
s_{C_{p}}'(x) = i_{C_{p}} h_{C_{p}}(x) + s_{C_{p}}(x),
\end{equation}
\begin{equation}\label{eq12}
s_{e}'(a) = i_{e} h_{e}(a) + s_{e}(a).
\end{equation}

Now, for $x, y \in L(C_{p})$, by substituting \eqref{eq11} into the identity from \eqref{eq3} and applying Proposition \eqref{Prob2}, we obtain
\begin{align*}
i_{C_{p}} f_{C_{p}}(x, y) + s_{C_{p}}'[x, y]_{L} - i_{C_{p}} h_{C_{p}}[x, y]_{L}
&= [(s_{C_{p}}'(x) - i_{C_{p}} h_{C_{p}}(x)), (s_{C_{p}}'(y) - i_{C_{p}} h_{C_{p}}(y))]_{B} \\
&= [s_{C_{p}}'(x), s_{C_{p}}'(y)]_{B}
- [s_{C_{p}}'(x), i_{C_{p}} h_{C_{p}}(y)]_{B}
- [i_{C_{p}} h_{C_{p}}(x), s_{C_{p}}'(y)]_{B} \\
&\quad + [i_{C_{p}} h_{C_{p}}(x), i_{C_{p}} h_{C_{p}}(y)]_{B}.
\end{align*}

Since the last bracket vanishes, the expression simplifies to
\begin{align*}
i_{C_{p}} f_{C_{p}}(x, y) + s_{C_{p}}'[x, y]_{L} - i_{C_{p}} h_{C_{p}}[x, y]_{L}
&= s_{C_{p}}'[x, y]_{L} + i_{C_{p}} g_{C_{p}}(x, y)
- [s_{C_{p}}'(x), i_{C_{p}} h_{C_{p}}(y)]_{B}
+ [i_{C_{p}} h_{C_{p}}(x), s_{C_{p}}'(y)]_{B}.
\end{align*}

Hence,
\[
g_{C_{p}}(x, y) - f_{C_{p}}(x, y) = x \cdot h_{C_{p}}(y) - h_{C_{p}}([x, y]_{L}) - y \cdot h_{C_{p}}(x),
\]
which implies
\[
\delta h_{C_{p}}(x, y) = g_{C_{p}}(x, y) - f_{C_{p}}(x, y).
\]

Therefore, $f_{C_{p}}$ and $g_{C_{p}}$ differ by a Chevalley–Eilenberg 2-coboundary in the category of $C_{p}$-Green functors of Lie type.

Similarly, for $a_{1}, b_{1} \in L(e)$, substituting \eqref{eq12} into the identity from \eqref{eq4} and applying Proposition \eqref{Prob2}, we get
\begin{align*}
i_{e} f_{e}(a_{1}, b_{1}) + s_{e}'[a_{1}, b_{1}]_{L} - i_{e} h_{e}[a_{1}, b_{1}]_{L}
&= [(s_{e}'(a_{1}) - i_{e} h_{e}(a_{1})), (s_{e}'(b_{1}) - i_{e} h_{e}(b_{1}))]_{B} \\
&= [s_{e}'(a_{1}), s_{e}'(b_{1})]_{B}
- [s_{e}'(a_{1}), i_{e} h_{e}(b_{1})]_{B}
- [i_{e} h_{e}(a_{1}), s_{e}'(b_{1})]_{B}.
\end{align*}

Hence,
\[
g_{e}(a_{1}, b_{1}) - f_{e}(a_{1}, b_{1}) = a_{1} \cdot h_{e}(b_{1}) - h_{e}([a_{1}, b_{1}]_{L}) - b_{1} \cdot h_{e}(a_{1}),
\]
which gives
\[
\delta h_{e}(a_{1}, b_{1}) = g_{e}(a_{1}, b_{1}) - f_{e}(a_{1}, b_{1}).
\]

Therefore, $f_{e}$ and $g_{e}$ differ by a Chevalley–Eilenberg 2-coboundary for $C_{p}$-Green functors of Lie type. Thus, there exists a well-defined map from $\operatorname{Ext}(L, M)$ to $H^{2}_{\operatorname{Lie}_{(C_{p})}}(L, M)$.

\medskip

\noindent \textbf{(b)}
Let $[f_{C_{p}}]$ and $[f_{e}] \in H^{2}_{\operatorname{Lie}_{(C_{p})}}(L, M)$, where $f_{C_{p}}$ and $f_{e} \in Z^{2}_{\operatorname{Lie}_{(C_{p})}}(L, M)$. We define the $C_{p}$-Green functor of Lie type $B_{f_{C_{p}}, f_{e}}$ as in Proposition \eqref{Prob1}. Hence, the extension associated to $f_{C_{p}}$ and $f_{e}$ is
\[
\begin{tikzcd}
E_{f_{C_{p}}, f_{e}}: 0 \arrow{r} & M \arrow{r}{i} & B_{f_{C_{p}}, f_{e}} \arrow{r}{j} & L \arrow{r} & 0,
\end{tikzcd}
\]
where $j$ is a homomorphism of $C_{p}$-Green functors of Lie type and $i$ is a homomorphism of $C_{p}$-Mackey functors.

We now show that $[E_{f_{C_{p}}}]$ and $[E_{f_{e}}]$ are independent of the choice of $f_{C_{p}}$ and $f_{e}$. That is, if $[f_{C_{p}}] = [g_{C_{p}}]$, then $f_{C_{p}} = g_{C_{p}} + \delta h_{C_{p}}$, and similarly $[f_{e}] = [g_{e}]$ implies $f_{e} = g_{e} + \delta h_{e}$.

Two extensions $E_{f_{C_{p}}}$ and $E_{g_{C_{p}}}$ are equivalent if and only if there exists a commutative diagram
\[
\begin{tikzcd}[row sep=large, column sep=large]
0 \ar[r] & M(C_{p}) \ar[r, "i_{C_{p}}"] \ar[d, equal] & B_{f_{C_{p}}}(C_{p}) \ar[d, "\beta_{C_{p}}"] \ar[r, "j_{C_{p}}"] & L(C_{p}) \ar[r] \ar[d, equal] & 0 \\
0 \ar[r] & M(C_{p}) \ar[r, "i'_{C_{p}}"] & B_{g_{C_{p}}}(C_{p}) \ar[r, "j'_{C_{p}}"] & L(C_{p}) \ar[r] & 0
\end{tikzcd}
\]
with $\beta_{C_{p}}$ a Lie algebra homomorphism. Commutativity implies that
\[
\beta_{C_{p}}(u, x) = (u + h_{C_{p}}(x), x)
\]
for some $h_{C_{p}} \in C^{1}_{C_{p}}(L, M)$. The fact that $\beta_{C_{p}}$ is a Lie homomorphism gives the identity
\begin{equation}\label{eq13}
\beta_{C_{p}}([(u, x), (v, y)]_{B_{f_{C_{p}}}}) = (xv - yu + f_{C_{p}}(x, y) + h_{C_{p}}([x, y]_{L}), [x, y]_{L}),
\end{equation}
and
\begin{equation}\label{eq14}
[\beta_{C_{p}}(u, x), \beta_{C_{p}}(v, y)]_{B_{g_{C_{p}}}} = (xv + x h_{C_{p}}(y) - yu - y h_{C_{p}}(x) + g_{C_{p}}(x, y), [x, y]_{L}).
\end{equation}
Therefore, from equations~\eqref{eq13} and~\eqref{eq14}, we obtain
\[
\delta h_{C_{p}}(x, y) = f_{C_{p}}(x, y) - g_{C_{p}}(x, y) = x h_{C_{p}}(y) - h_{C_{p}}([x, y]_{L}) - y h_{C_{p}}(x).
\]
Hence, $f_{C_{p}} - g_{C_{p}}$ is a Chevalley–Eilenberg 2-coboundary of the $C_{P}$-Green functors of Lie type.

\bigskip

Conversely, suppose $f_{C_{p}} - g_{C_{p}}$ is a Chevalley–Eilenberg 2-coboundary of the $C_{P}$-Green functors of Lie type. Then we aim to prove that the extensions $E_{f_{C_{p}}}$ and $E_{g_{C_{p}}}$ are equivalent. Two extensions $E_{f_{e}}$ and $E_{g_{e}}$ are said to be equivalent if and only if there exists a commutative diagram:
\[
\begin{tikzcd}[row sep=large, column sep=large]
0 \ar[r] & M(e) \ar[r, "i_{e}"] \ar[d, equal] & B_{f_{e}}(e) \ar[d, "\beta_{e}"] \ar[r, "j_{e}"] & L(e) \ar[r] \ar[d, equal] & 0 \\
0 \ar[r] & M(e) \ar[r, "i'_{e}"] & B_{g_{e}}(e) \ar[r, "j'_{e}"] & L(e) \ar[r] & 0,
\end{tikzcd}
\]
where $\beta_{e}$ is a homomorphism of Lie algebras. The commutativity of the diagram implies that
\[
\beta_{e}(n, b) = (n + h_{e}(b), b)
\]
for some $h_{e} \in C^{1}_{e}(L, M)$. The fact that $\beta_{e}$ is a Lie algebra homomorphism yields the identity
\begin{equation} \label{eq15}
\beta_{e}([(n, b), (m, a)]_{B_{f_{e}}}) 
= \beta_{e}(b m - a n + f_{e}(b, a), [b, a]_{L}) 
= (b m - a n + f_{e}(b, a) + h_{e}([b, a]_{L}), [b, a]_{L}),
\end{equation}
and
\begin{equation} \label{eq16}
[\beta_{e}(n, b), \beta_{e}(m, a)]_{B_{g_{e}}} 
= [(n + h_{e}(b), b), (m + h_{e}(a), a)]_{B_{g_{e}}} 
= (b m + b h_{e}(a) - a n - a h_{e}(b) + g_{e}(b, a), [b, a]_{L}).
\end{equation}
Therefore, from equations~\eqref{eq15} and~\eqref{eq16}, we obtain
\[
\delta h_{e}(b, a) = f_{e}(b, a) - g_{e}(b, a) = b h_{e}(a) - h_{e}([b, a]_{L}) - a h_{e}(b).
\]
Hence, $f_{e} - g_{e}$ is a Chevalley–Eilenberg 2-coboundary of the $C_{P}$-Green functors of Lie type.

\medskip

Conversely, suppose $f_{e} - g_{e}$ is a Chevalley–Eilenberg 2-coboundary of the $C_{P}$-Green functors of Lie type. Then we aim to prove that the extensions $E_{f_{e}}$ and $E_{g_{e}}$ are equivalent. Moreover, we need to verify that the following diagram commutes:
\[
\begin{tikzcd}
0 & M(C_{p}) && B_{f_{C_{p}}}(C_{p}) && L(C_{p}) && 0 \\
0 && M(C_{p}) && B_{g_{C_{p}}}(C_{p}) && L(C_{p}) & 0 \\
0 & M(e) && B_{f_{e}}(e) && L(e) && 0 \\
0 && M(e) && B_{g_{e}}(e) && L(e) & 0 \\
\arrow[from=1-1, to=1-2]
\arrow["i_{C_{p}}", from=1-2, to=1-4]
\arrow["r_{M}"{description, pos=0.7}, curve={height=-18pt}, from=1-2, to=3-2]
\arrow["j_{C_{p}}", from=1-4, to=1-6]
\arrow["\beta_{C_{p}}"{pos=0.6}, from=1-4, to=2-5]
\arrow["r_{B}"{description, pos=0.7}, curve={height=-18pt}, from=1-4, to=3-4]
\arrow[from=1-6, to=1-8]
\arrow[equals, from=1-6, to=2-7]
\arrow["r_{L}"{description, pos=0.7}, curve={height=-18pt}, from=1-6, to=3-6]
\arrow[from=2-1, to=2-3]
\arrow[equals, from=2-3, to=1-2]
\arrow["i^{'}_{C_{p}}", from=2-3, to=2-5]
\arrow["r_{M}"{description, pos=0.7}, curve={height=-18pt}, from=2-3, to=4-3]
\arrow["j^{'}_{C_{p}}", from=2-5, to=2-7]
\arrow["r_{B}"{description, pos=0.7}, curve={height=-18pt}, from=2-5, to=4-5]
\arrow[from=2-7, to=2-8]
\arrow["r_{L}"{description, pos=0.7}, curve={height=-18pt}, from=2-7, to=4-7]
\arrow[from=3-1, to=3-2]
\arrow["t_{M}"{description, pos=0.7}, curve={height=-18pt}, from=3-2, to=1-2]
\arrow["i_{e}", from=3-2, to=3-4]
\arrow[equals, from=3-2, to=4-3]
\arrow["t_{B}"{description, pos=0.7}, curve={height=-18pt}, from=3-4, to=1-4]
\arrow["j_{e}", from=3-4, to=3-6]
\arrow["\beta_{e}"'{pos=0.4}, shorten >=6pt, from=3-4, to=4-5]
\arrow["t_{L}"{description, pos=0.7}, curve={height=-18pt}, from=3-6, to=1-6]
\arrow[from=3-6, to=3-8]
\arrow[equals, from=3-6, to=4-7]
\arrow[from=4-1, to=4-3]
\arrow["t_{M}"{description, pos=0.7}, curve={height=-18pt}, from=4-3, to=2-3]
\arrow["i^{'}_{e}", from=4-3, to=4-5]
\arrow["t_{B}"{description, pos=0.7}, curve={height=-18pt}, from=4-5, to=2-5]
\arrow["j^{'}_{e}", from=4-5, to=4-7]
\arrow["t_{L}"{description, pos=0.7}, curve={height=-18pt}, from=4-7, to=2-7]
\arrow[from=4-7, to=4-8]
\end{tikzcd}
\]

It suffices to check that \( r_{B} \circ \beta_{C_{p}}(u, x) = \beta_{e} \circ r_{B}(u, x) \). We compute:
\[
r_{B} \circ \beta_{C_{p}}(u, x) = r_{B}(u + h_{C_{p}}(x), x) = \big(r_{M}(u) + r_{M}(h_{C_{p}}(x)),\, r_{L}(x)\big),
\]
and
\[
\beta_{e} \circ r_{B}(u, x) = \beta_{e}(r_{M}(u), r_{L}(x)) = \big(r_{M}(u) + h_{e}(r_{L}(x)),\, r_{L}(x)\big).
\]

Thus, by Definition \eqref{Def1}, it follows that \( r_{B} \circ \beta_{C_{p}} = \beta_{e} \circ r_{B} \). Therefore, we conclude that there exists a well-defined map
\[
H^{2}_{\mathrm{Lie}_{(C_{p})}}(L, M) \longrightarrow \operatorname{Ext}(L, M).
\]

\medskip

\noindent \textbf{(c)}
 Let \( f_{C_p} \) and \( f_e \) be two Chevalley–Eilenberg 2-cocycles of \( C_p \)-Green functors of Lie type. We define the Lie bracket on \( B_{f_{C_p}}(C_p) \) and \( B_{f_e}(e) \) as follows:
\[
[(u,x),(v,y)]_B = \left(xv - yu + f_{C_p}(x,y), [x,y]_L\right),
\]
\[
[(m,a),(n,b)]_B = \left(an - bm + f_e(a,b), [a,b]_L\right),
\]
where \( u,v \in M(C_p) \), \( x,y \in L(C_p) \), \( m,n \in M(e) \), and \( a,b \in L(e) \).

The Chevalley–Eilenberg 2-cocycle of $C_{p}$-Green functor of Lie type conditions for \( f_{C_p} \) and \( f_e \) ensure that these brackets satisfy the Jacobi identity and antisymmetry. Hence, \( B_{f_{C_p}}(C_p) \) and \( B_{f_e}(e) \) are Lie algebras.

We now define the maps
\begin{align*}
    i_{C_p} & : M(C_p) \longrightarrow B_{f_{C_p}}(C_p), & i_{C_p}(u) &= (u,0), \\
    i_e     & : M(e) \longrightarrow B_{f_e}(e),         & i_e(m)     &= (m,0), \\
    j_{C_p} & : B_{f_{C_p}}(C_p) \longrightarrow L(C_p), & j_{C_p}(u,x) &= x, \\
    j_e     & : B_{f_e}(e) \longrightarrow L(e),         & j_e(m,a)     &= a.
\end{align*}

Here, \( i_{C_p} \) and \( i_e \) are homomorphisms of \( C_p \)-Mackey functors, and \( j_{C_p} \) and \( j_e \) are homomorphisms of \( C_p \)-Green functors of Lie type.

Thus, the sequence
$$
E_{f_{C_{p}},f_{e}} : 0 \to M \overset{i}{\to} B \overset{j}{\underset{s}{\rightleftarrows}} L \to 0
$$
is exact.

For any \( x \in L(C_p) \) and \( a \in L(e) \), define the sections \( s_{C_p}(x) = (0,x) \) and \( s_e(a) = (0,a) \). Then, for \( x,y \in L(C_p) \), we compute
\begin{align*}
[s_{C_p}(x), s_{C_p}(y)]_B 
&= [(0,x), (0,y)]_B \\
&= (f_{C_p}(x,y), [x,y]_L) \\
&= (f_{C_p}(x,y), 0) + (0, [x,y]_L) \\
&= i_{C_p}(f_{C_p}(x,y)) + s_{C_p}([x,y]_L).
\end{align*}

Similarly, for \( a,b \in L(e) \), we have
\begin{align*}
[s_e(a), s_e(b)]_B 
&= [(0,a), (0,b)]_B \\
&= (f_e(a,b), [a,b]_L) \\
&= (f_e(a,b), 0) + (0, [a,b]_L) \\
&= i_e(f_e(a,b)) + s_e([a,b]_L).
\end{align*}
Therefore, we conclude that the chosen sections \( s_{C_p} \) and \( s_e \) yield the Chevalley–Eilenberg 2-cocycles \( f_{C_p} \) and \( f_e \) for the \( C_p \)-Green functors of Lie type.

\medskip

Conversely, suppose the complex
\[
E : 0 \longrightarrow M \xrightarrow{i} B \xrightarrow{j} L \longrightarrow 0
\]
is an extension, and let \( f_{C_p} \) and \( f_e \) be the Chevalley–Eilenberg 2-cocycles of \( C_p \)-Green functors of Lie type obtained from this extension. We now show that the extension
\[
E_{f_{C_p}, f_e} : 0 \longrightarrow M \xrightarrow{i} B_{f_{C_p}, f_e} \xrightarrow{j} L \longrightarrow 0
\]
associated with \( f_{C_p} \) and \( f_e \) is equivalent to the given extension \( E \). 

The extensions \( E \) and \( E_{f_{C_p}, f_e} \) are equivalent if there exists a homomorphism 
\[
\tau_{f_{C_p}, f_e} : B_{f_{C_p}, f_e} \longrightarrow B
\]
such that the following diagram commutes:
\[
\begin{tikzcd}[row sep=large, column sep=large]
0 \ar[r] & M \ar[r, "i"] \ar[d, equal] & B \ar[r, shift left, "j"] & L \ar[r] \ar[l, shift left, "s"] \ar[d, equal] & 0 \\
0 \ar[r] & M \ar[r, "i"] & B_{f_{C_p}, f_e} \ar[u, "\tau_{f_{C_p}, f_e}"] \ar[r, "j"] & L \ar[r] & 0
\end{tikzcd}
\]

From the commutativity of the diagram, we have
\[
\tau_{f_{C_p}}(u, x) = i_{C_p}(u) + s_{C_p}(x), \quad \text{and} \quad \tau_{f_e}(m, a) = i_e(m) + s_e(a),
\]
where \( u \in M(C_p) \), \( x \in L(C_p) \), \( m \in M(e) \), and \( a \in L(e) \).

It remains to verify that the maps \( \tau_{f_{C_p}} \) and \( \tau_{f_e} \) are Lie algebra homomorphisms.

Let \( (u,x), (v,y) \in B(C_p) \). Then we compute
\begin{align*}
\tau_{f_{C_p}}([(u,x),(v,y)]_{B_{f_{C_p}, f_e}}) 
&= \tau_{f_{C_p}}(xv - yu + f_{C_p}(x,y), [x,y]_L) \\
&= i_{C_p}(xv) - i_{C_p}(yu) + i_{C_p}(f_{C_p}(x,y)) + s_{C_p}([x,y]_L) \\
&= [s_{C_p}(x), i_{C_p}(v)]_B + [i_{C_p}(u), s_{C_p}(y)]_B + [s_{C_p}(x), s_{C_p}(y)]_B.
\end{align*}

On the other hand, we have
\begin{align*}
[\tau_{f_{C_p}}(u,x), \tau_{f_{C_p}}(v,y)]_B 
&= [(i_{C_p}(u) + s_{C_p}(x)), (i_{C_p}(v) + s_{C_p}(y))]_B \\
&= [i_{C_p}(u), i_{C_p}(v)]_B + [s_{C_p}(x), i_{C_p}(v)]_B + [i_{C_p}(u), s_{C_p}(y)]_B + [s_{C_p}(x), s_{C_p}(y)]_B \\
&= [s_{C_p}(x), i_{C_p}(v)]_B + [i_{C_p}(u), s_{C_p}(y)]_B + [s_{C_p}(x), s_{C_p}(y)]_B.
\end{align*}

Thus, \( \tau_{f_{C_p}} \) is a Lie homomorphism.

Similarly, let \( (m,a), (n,b) \in B(e) \). Then
\begin{align*}
\tau_{f_e}([(m,a),(n,b)]_{B_{f_{C_p}, f_e}})
&= \tau_{f_e}(an - bm + f_e(a,b), [a,b]_L) \\
&= i_e(an) - i_e(bm) + i_e(f_e(a,b)) + s_e([a,b]_L) \\
&= [s_e(a), i_e(n)]_B + [i_e(m), s_e(b)]_B + [s_e(a), s_e(b)]_B.
\end{align*}

And
\begin{align*}
[\tau_{f_e}(m,a), \tau_{f_e}(n,b)]_B 
&= [(i_e(m) + s_e(a)), (i_e(n) + s_e(b))]_B \\
&= [i_e(m), i_e(n)]_B + [s_e(a), i_e(n)]_B + [i_e(m), s_e(b)]_B + [s_e(a), s_e(b)]_B \\
&= [s_e(a), i_e(n)]_B + [i_e(m), s_e(b)]_B + [s_e(a), s_e(b)]_B.
\end{align*}

Therefore, \( \tau_{f_{C_p}, f_e} \) is a Lie algebra homomorphism.

\medskip

Hence, the theorem is proved.

\end{proof}

 \begin{center}
 {\bf ACKNOWLEDGEMENT}
 \end{center}
 
This research is sponsored and supported by the Core Research Grant (CRG) of the Anusandhan National Research Foundation (ANRF), formerly the Science and Engineering Research Board (SERB), under the Department of Science and Technology (DST), Government of India (Grant Number: CRG/2022/005332). All authors gratefully acknowledge the project grant received from the aforementioned agency.

\end{document}